\documentclass[a4paper,10pt]{amsart}
\usepackage[utf8]{inputenc}
\usepackage{nicefrac,amsmath,amssymb,amsthm, multirow, array, graphicx} 

\usepackage{mathtools}
\usepackage{tikz}
\usepackage{tikzsymbols}
\tikzstyle{ball} = [circle,shading=ball,ball color=black,minimum size=1mm,inner sep=1.3pt]
\usepackage{epstopdf}
\usepackage{framed}
\usepackage{tikz}
\usepackage{todonotes}
\usepackage{enumitem}
\usepackage{hyperref}
\usepackage{cleveref}
\usepackage{cancel}
\usepackage{array}
\newcolumntype{L}{>{$}l<{$}} 

\allowdisplaybreaks

\newtheorem{Theo}[subsection]{Theorem}
\newtheorem{TheoM}{Theorem}
\crefname{Theo}{Theorem}{Theorems}

\newtheorem*{defi-intro}{Definition}
\newtheorem*{Teo}{Theorem}

\newtheorem*{THM-KuttlerReichstein}{Theorem 1.1 in \cite{KuttlerReichstein-Intrinsic}}

\newtheorem{lem}[subsection]{Lemma}
\newtheorem{pro}[subsection]{Proposition}
\newtheorem{Cor}[subsection]{Corollary}
\theoremstyle{definition}
\newtheorem{defi}[subsection]{Definition}

\theoremstyle{remark}

\newtheorem*{acknowledgement}{Acknowledgements}

\newtheorem{Question}{Question}
\crefname{Question}{Question}{Questions}

\def\subweyl#1{\ensuremath{\mathcal R_1#1}}

\def\alt{\ensuremath{\mathsf{Alt}}}
\def\mmid{\,\middle\vert\,}

\def\J{\ensuremath J}
\newcommand{\triv}[1][1]{\mathsf 0_{#1}}

\def\diag[#1,#2]{\ensuremath{\begin{pmatrix} {#1}&0\\ 0&{#2}\end{pmatrix}}}

\def\define{\ensuremath{\overset{\operatorname{\scriptscriptstyle def}}=}}

\definecolor{cof}{RGB}{219,144,71}
\definecolor{pur}{RGB}{186,146,162}
\definecolor{greeo}{RGB}{91,173,69}
\definecolor{greet}{RGB}{52,111,72}

\numberwithin{equation}{section} 
\numberwithin{subsection}{section} 

\makeatletter
\let\c@equation\c@subsection
\makeatother

\makeatletter
\let\c@table\c@figure 
\makeatother

\title[Finite Groups Generated in Low Real Codimension]{Finite Groups Generated\\ in Low Real Codimension}
\author{Ivan Martino}
\author{Rahul Singh}

\date{\today}

\begin{document}
\begin{abstract}

We study the intersection lattice of the arrangement $\mathcal{A}^G$ of subspaces fixed by subgroups of a finite linear group $G$.
When $G$ is a reflection group, this arrangement is precisely the hyperplane reflection arrangement of $G$.
We generalize the notion of finite reflection groups. 
We say that a group $G$ is generated (resp. strictly generated) in codimension $k$ if it is generated by its elements that fix point-wise a subspace of codimension at most $k$ (resp. precisely $k$). 

If $G$ is generated in codimension two, we show that the intersection lattice of $\mathcal{A}^G$ is atomic. We prove that the alternating subgroup $\alt(W)$ of a reflection group $W$ is strictly generated in codimension two; moreover, the subspace arrangement of $\alt(W)$ is the truncation at rank two of the reflection arrangement $\mathcal{A}^W$.

Further, we compute the intersection lattice of all finite subgroups of $GL_3(\mathbb{R})$, and moreover, we emphasize the groups that are ``minimally generated in real codimension two", i.e, groups that are strictly generated in codimension two but have no real reflection representations. 
We also provide several examples of groups generated in higher codimension.
\end{abstract}
\maketitle

\setcounter{section}{0}\setcounter{subsection}{0}

Let $G$ be a finite subgroup of $GL(V)$, $G\xhookrightarrow{\rho} GL(V)$, and consider the quotient map $\pi:V\rightarrow\nicefrac{V}{G}=:X$. 
In this paper, we study the Luna stratification \cite{Luna-Slices} of $\nicefrac{V}{G}$. 
In this case, the Luna stratification coincides with the isotropy stratification: each stratum $X_{\mathcal{H}}$ of $\nicefrac{V}{G}$ consists of the irreducible component of closed orbits having isotropy group in a specified conjugacy class $\mathcal{H}$ of subgroups of $G$. 

In other words, if $v$ is a vector of $V$ then \[
    \operatorname{Stab}_\rho(v)\define \{g\in G \mid g v= v\}
\] 
is the isotropy group at $v$.
We denote by $V_H^{\rho}$ the locus of points in $V$ whose isotropy group is precisely $H$ 
\[
    V_H^{\rho}\define \{v\in V\mid H= \operatorname{Stab}_{\rho}(v)\}
\]
and by $V_{\mathcal{H}}^{\rho}$ the union of $V_L^{\rho}$ for all $L$ in the conjugacy class $\mathcal{H}$, i.e.,
\[
	V_{\mathcal{H}}^{\rho}\define \bigcup_{L\in \mathcal{H}} V_L^{\rho}.
\]
If $V_H^{\rho}\neq \emptyset$, than we say $H$ is a \emph{stabilizer subgroup} with respect to $\rho$. 
%
The quotient map $\pi$ restricted to $V_{\mathcal{H}}^{\rho}$ is surjective onto $X_{\mathcal{H}}$; further, $\pi\mid_{V_{H}^{\rho}} :V_{H}^{\rho}\rightarrow X_{\mathcal{H}}$ is a principal $(\nicefrac{N_G(H)}{H})$-bundle where $N_G(H)$ is the normalizer of $H$ in $G$, see \cite[\S6.9]{VinbergPopov-Invariant} and \cite[\S1.5]{Schwarz-Lifting}. 

As $H$ varies among the stabilizer subgroups, the subvarieties $V_H^{\rho}$ from a stratification. 
The closure of each stratum $\overline{V_H^{\rho}}$ is the union of open strata, i.e.,\[
    V^H_{\rho}\define \overline{V_H^{\rho}}=\bigsqcup\limits_{H\subset H'} V_{H'}^{\rho},
\]
where the disjoint union is over the stabilizer subgroups $H'$ that contain $H$. 
It is easy to check that the closed strata are linear subspaces of $V$ and that 
\[
	V^H_{\rho}=\{v\in V\mid H \subseteq \operatorname{Stab}_{\rho}(v)\}.
\]
The main goal of this article is to study the collection $\mathcal{A}^{\rho}$ of these subspaces.
\begin{defi-intro}
The arrangement of subspaces {of the representation $\rho$} is 
  \[
    \mathcal{A}^{\rho}\define \{V^{H}_{\rho} \mid \{e\}\neq H \mbox{ is a stabilizer subgroup w.r.t. } \rho\}.
  \]
\end{defi-intro}

\noindent
We remark that this arrangement depends by the representation $\rho$ or, equivalently, it is associated to the linear group $G\xhookrightarrow{\rho} GL(V)$ and not to an abstract group.

The \emph{principal stratum} $V_e^\rho$ of the Luna stratification is the stratum with trivial isotropy group.
It is the complement of the arrangement $\mathcal A^\rho$, 
\begin{equation*}
	V_e^\rho = V\setminus \bigcup_{V^H\in \mathcal{A}^{\rho}} V^H.
\end{equation*}
This is a very well studied object in literature when $\rho$ is a reflection representation \cite{Bessis-Kp1,Bessis-Quotients, Brieskorn-Sur-les-groupes, Deligne-Les-immeubles, Falk-Terao}.

\noindent
A reflection $r\in GL(V)$ is a finite order element whose fixed point set $V^{\langle r\rangle}$ is a hyperplane, i.e., $V^{\langle r\rangle}$ has codimension one.
We say that $W\xhookrightarrow{\rho} GL(V)$ is a \emph{reflection representation} if the finite linear subgroup $W$ is generated by its reflections, i.e., $W$ is a reflection subgroup.

The subspace arrangement $\mathcal{A}^{W}\define \mathcal{A}^{\rho}$ of a reflection representation has been studied extensively (see for instance \cite{DeConciniProcesi-Topics-book,Orlik-Terao-book}); it is a collection of hyperplanes, called the reflection hyperplane arrangement.
The maximal subspaces in $\mathcal{A}^{W}$ are precisely the reflection hyperplanes $V^{\langle r\rangle}$, for $r$ a reflection in $W$.

In this work, we study a class of finite groups that naturally generalizes the above description.
We focus on finite linear groups $G$ which are generated by elements fixing subspaces of codimension one or two.

Before introducing the definition of \emph{the groups (strictly) generated in codimension $i$}, and before going further with the presentation of our results, let us motivate this effort.

\subsubsection*{Intrinsic stratification.}
Kuttler and Reichstein \cite{KuttlerReichstein-Intrinsic} studied if every automorphism of $\nicefrac{V}{G}$ maps a Luna stratum to another stratum. When this happens, the stratification is said to be \emph{intrinsic}. 
In other words, a Luna stratification of $X$ is intrinsic if for every automorphism $f$ of $X$ and every conjugacy class $\mathcal H$ of stabilizer subgroups in $G$, we have $f(X_{\mathcal{H}})=X_{\mathcal{H'}}$ for some conjugacy class $\mathcal H'$ of stabilizer subgroups in $G$. 
Many reflection representations give rise to non-intrinsic stratifications; on the other hand if one removes the reflections, the Luna stratification is intrinsic.

\begin{Teo}
(cf. \cite[Theorem 1.1]{KuttlerReichstein-Intrinsic})
Let $G$ be a finite subgroup of ${GL}(V)$. Assume that $G$ does not contain reflections. 
Then, the Luna stratification of $\nicefrac VG$ is intrinsic. 
\end{Teo}

Kuttler and Reichstein \cite{KuttlerReichstein-Intrinsic} also exhibited several positive examples of linear reductive groups with intrinsic Luna stratification.
Later, Schwarz \cite{Schwarz-VectorLuna} proved that only finitely many connected simple groups $G$ have stratifications that fail to be intrinsic.
More results about this can be also found in \cite{Kuttler-Lifting}.

\subsubsection*{Mixed Tate property}
Ekedahl \cite{Ekedahl-inv} used the arrangement $\mathcal{A}^G$ to compute the class of the classifying stack $\mathcal{B} G$ for certain finite groups $G$, for instance the cyclic group $\nicefrac{\mathbb Z}{n\mathbb Z}$ and the Symmetric group $S_n$.
Later, understanding the Luna stratification of $\nicefrac{V}{G}$ was instrumental towards finding a recursive formula \cite{Martino-TEIFFG} for geometric invariants belonging to the split Grothendieck group of abelian groups $\operatorname{L_0}(Ab)$, see \cite{Intro-Ekedahl-Invariants}.
The combinatorics of Ekedahl computations \cite{Ekedahl-inv} is also the main interest of \cite{DelucchiMartinoBW}.

Totaro has used the combinatorics of $\mathcal{A}^G$, to show recursively that certain quotient substacks of $[\nicefrac{V}{G}]$, and the classifying stack $\mathcal{B} G$, are mixed Tate, see \cite[Section 9]{Totaro-MotiveClassifying}.

\subsubsection*{The open complement $V_e^{\rho}$}
Bessis \cite{Bessis-Kp1} has shown that the open complement $V_e^{\rho}$ of $\mathcal{A}^W$ is $\operatorname{K}(\pi,1)$. 
The principal stratum plays a crucial role even when the representation is not a reflection representation.
Indeed, a common theme in invariant theory is that for sufficiently large $s$ the tensor product representation $V^s$ describes the action of $G$ better than the representation $V$ itself.
In the most ambiguous way possible, this is due to the fact that $(V^s)_e$ is a \emph{nicer} principal open stratum than $V_e$.

There are many examples of this principle. 
For instance the class of the classifying stack $\mathcal{B} G$ in the Motivic ring of algebraic variety $\widehat{\operatorname{K}_0}(Var_{\mathbf{k}})$ can be approximated as 
\[
\{\mathcal{B} G\}=\lim_{s} \{\nicefrac{V^s}{G}\}\mathbb{L}^{-s\operatorname{dim}V}\in \widehat{\operatorname{K}_0}(Var_{\mathbf{k}})
\]
where $\mathbb{L}$ is the Lefschetz class in $\widehat{\operatorname{K}_0}(Var_{\mathbf{k}})$, see \cite[Proposition 2.5]{Intro-Ekedahl-Invariants}.
Other examples can be found in \cite{Lorenz-Cohen-Macaulay,Popov-GenericallyMultiple,KuttlerReichstein-Intrinsic}.

We also mention that whenever $V\setminus V_e^{\rho}$ has codimension at least two, the principal stratum $\nicefrac{V_e^{\rho}}{G}$ can be used as a model for the classifying space $\mathcal BG$ and, therefore, the group cohomology of $G$ and the unramified cohomology of $\mathcal BG$ can be computed using $\nicefrac{V_e}{G}$, see \cite{Totaro-MotiveClassifying} and \cite[Part 1, Appendix C]{SkipMerkurjevSerr-Cohomological}.

\subsubsection*{Linear Coxeter Groups and Alternating Subgroups} 
Any finite reflection group can be equipped with a Coxeter system \cite{MR1503182}.
For a linear finite reflection groups $W$, the kernel of the determinant map $\det:W\rightarrow\mathbb R$ is called the alternating subgroup $\alt(W)$ of $W$. 
The combinatorics of reflection groups, Coxeter groups and their alternating subgroups is well-studied, see \cite{MR1890629,BrentiReinerRoichman-Alternating}.
In this work we study $\alt(W)$ from a different point of view, see \Cref{pro-alternating-subgroup}.


%

%
%
%

\section*{Main Definitions}


As the reader might sense, in this work we generalize reflection groups to groups that fit the framework of the literature above.
We assume from now on that $G$ is a finite linear group with representation $G\xhookrightarrow{\rho} GL(V)$ for some finite vector space $V$.

More generally, all groups (including Coxeter groups) are implicitly assumed to be finite linear groups; otherwise we call them \emph{abstract groups} and we write $|G|$ for the underlying abstract group of $G$.


\begin{defi-intro}
Let $g$ be a finite order element of ${GL}(V)$. 
We say that $g$ is generated in codimension $i$ if $\operatorname{codim}V^{\langle g\rangle}=i$.
\end{defi-intro}
Reflections (resp. rotations) in $GL(V)$ are precisely the elements generated in codimension one (resp. two).
Observe that if $V$ is a real vector space, the reflections have order two. 

We now introduce the main object of this work.
\begin{defi-intro}
We say that a group $G$ is generated in codimension $k$ if 
$G$ is generated by elements of codimension at most $k$, i.e.,
\[
	G=\left\langle g\mmid \operatorname{codim}{V^{\langle g\rangle}}\leq k\right\rangle
\]
\end{defi-intro}

It is clear that groups generated in codimension one are precisely the finite Coxeter groups.
%

\begin{defi-intro}
We say that a group $G$ is \emph{strictly} generated in codimension $k$ if $G=\left\langle g\mmid \operatorname{codim}V^{\langle g\rangle} \pmb{=} k\right\rangle$.
\end{defi-intro}
We show in \Cref{pro-alternating-subgroup} that alternating subgroups of Coxeter groups are strictly generated in codimension two.
In fact, we verify that in dimensions up to three, any group that strictly generated in codimension two is the alternating subgroup of some Coxeter group; and further, we can recover the Coxeter group from its alternating group, i.e., if $\alt(W_1)=\alt(W_2)$ (as linear groups), then $W_1=W_2$ (see \Cref{allGroupData}).

In \Cref{pro:Klein-class}, we list the subgroups of $GL_3(\mathbb R)$ which are strictly generated in codimension two.
We present an infinite class of groups strictly generated in codimension two in $GL_{3n}(\mathbb{R})$ in \Cref{Pro-gen-codim-2-in-high-dimension}. 
Finally in \Cref{Pro-group-gen-in-codim-2n}, we present another class of groups strictly generated in codimension $2n$ in $GL_{3n}(\mathbb{R})$.

\begin{defi-intro}
An abstract group $G$ is \emph{minimally generated} in codimension $k$ if $k$ is the minimal integer for which there exists a faithful representation $\rho:G\rightarrow GL(V)$ such that the linear group $\rho(G)$ is strictly generated in codimension $k$.
\end{defi-intro}   

We study several interesting examples of abstract groups that are generated in codimension two, but not generated in codimension one, i.e., they have no reflection representation. 
For instance, the cyclic groups $\nicefrac{\mathbb Z}{n\mathbb Z}$, for $n\geq 3$ are minimally generated in (real) codimension two. 
In particular, they are not reflection groups. 
A representation $\nicefrac{\mathbb Z}{n\mathbb Z}\subseteq GL_2(\mathbb{R})$ is shown in \Cref{Cor-mu-2-n>2isNeatInCod2}.

\section*{Main Results.}
%
%
We work exclusively with real representations and we focus primarily on groups generated in codimension two.
We plan to study the complex case in \cite{MartinoSingh-complex}.

The perk of groups strictly generated in codimension two is that they save many of the nice combinatorial features of the reflection groups.
For instance, while the intersection lattice $\mathcal{L}(\mathcal{A}^{\rho})$ of the arrangements $\mathcal{A}^{\rho}$ is no longer geometric, it is still atomic, see \Cref{thm-atomic-lattice}. 
%
Further if $G$ is strictly generated in codimension two, the maximal subspaces in the arrangement are of codimension two, see \Cref{pro:no-cod-one}. 
This is not the case for groups strictly generated in higher codimensions and we provide examples in \Cref{sec:higherRealDimension}.

%
%

In \Cref{sec:rank-two}, we show that the `strictly generated in codimension two' subgroups of $GL_2(\mathbb R)$ are precisely the cyclic groups $\mu_2(n)$, see \Cref{prop-GL2-real}.

In \Cref{sec:classification,sec:rank-three}, we completely classify the arrangements $\mathcal{A}^{G}$ of finite linear subgroups $G$ in $GL_3(\mathbb{R})$. 
\Cref{fig:all-arrangements} shows the inclusion order on linear subgroups of $GL_3(\mathbb R)$ from left to right.

In our first result we list all the groups $G$ in $GL_3(\mathbb R)$, for which the underlying abstract group $|G|$ is minimally generated in codimension two.
In particular, we find that there exist abstract groups minimally generated in codimension two, besides the toy example $\nicefrac{\mathbb Z}{n\mathbb Z}$.

\begin{TheoM}
The subgroups of $GL_3(\mathbb{R})$ which are minimally generated in codimension two are precisely the following: the Cyclic group $\mu_3(n)$ for $n\geq 3$, the Tetrahedral group $T_3$, and the Icosahedral group $Ico_3$.
\end{TheoM}

Next, we study the intersection lattice of the subspace arrangement of a group strictly generated in codimension two in $GL_3(\mathbb{R})$. 

\begin{TheoM}
Suppose $G$ is a group strictly generated in codimension two in $GL_3(\mathbb{R})$. Then $\mathcal L(\mathcal{A}^G)$ is isomorphic, as a poset, to $\mathcal L_n$, see \Cref{fig:Ln}, for some $n\neq 2$.

More precisely, for any positive integer $n\neq 2$, there exists $G\subset {GL}_3(\mathbb{R})$, strictly generated in codimension two such that $\mathcal L(\mathcal{A}^G)\cong \mathcal L_n$.
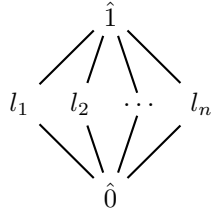
\begin{figure}[h]
\begin{tikzpicture}[thick, scale=0.8]
  \node (zero) at (0,1.5) {$\hat{1}$};
  \node (l0) at (-1.5, 0) {$l_1$};
  \node (l1) at (-0.5, 0) {$l_2$};
  \node (li) at (0.5, 0) {$\dots$};
  \node (ln) at (1.5, 0) {$l_n$};
  \node (whole) at (0,-1.5) {$\hat{0}$};
  \draw (whole) -- (l0) -- (zero);
  \draw (whole) -- (l1) -- (zero);
  \draw (whole) -- (li) -- (zero);
  \draw (whole) -- (ln) -- (zero);
\end{tikzpicture} 
\caption{We denote this poset as $\mathcal L_n$.}
\label{fig:Ln}.
\end{figure}
\end{TheoM}

The combinatorics of the quotient arrangement $\nicefrac{\mathcal A^G}G$ encodes the combinatorics of the Luna stratification. 

\begin{TheoM}
Suppose $G\subset GL_3(\mathbb R)$ is strictly generated in codimension two.
The intersection lattice of the closures of the Luna strata of $\nicefrac{V}{G}$ is isomorphic to $\mathcal L_n$ for $n\in\{2,3\}$. 
\begin{figure}[h]
\begin{tabular}{cc}
\begin{tikzpicture}[thick, scale=0.8]
  \node (zero) at (0,1.5) {$\hat{1}$};
  \node (l0) at (-1, 0) {$l_1$};
  \node (l1) at (+1, 0) {$l_2$};
  \node (whole) at (0,-1.5) {$\hat{0}$};
  \draw (whole) -- (l0) -- (zero);
  \draw (whole) -- (l1) -- (zero);
\end{tikzpicture} &
\begin{tikzpicture}[thick, scale=0.8]
  \node (zero) at (0,1.5) {$\hat{1}$};
  \node (l0) at (-1.5, 0) {$l_1$};
  \node (l1) at (0, 0) {$l_2$};
  \node (li) at (1.5, 0) {$l_3$};
  \node (whole) at (0,-1.5) {$\hat{0}$};
  \draw (whole) -- (l0) -- (zero);
  \draw (whole) -- (l1) -- (zero);
  \draw (whole) -- (li) -- (zero);
\end{tikzpicture} 
\end{tabular}
\caption{The two possible intersection lattice for the quotient arrangement $\nicefrac{V}{G}$, when $G$ is a finite subgroup of $GL_3(\mathbb R)$ and it is strictly generated in codimension two.}
\label{fig:Ln}
\end{figure}
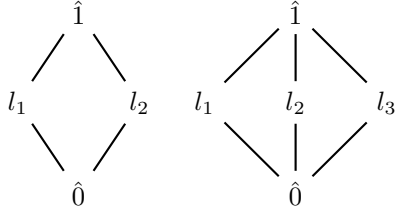
\end{TheoM}

The Luna stratification of $\nicefrac{V}{G}$ and the arrangement $\mathcal{A}^G$ are described for all subgroups of $GL_3(\mathbb R)$ in \Cref{sec:classification}.

Finally, as a byproduct we describe the cohomology of the open complement $U_G$ of the arrangement $\mathcal{A}^G$ in the vector space $V$, see \Cref{Thm-cohomology-GL3,THM:cohomology-dim4}.

\begin{TheoM}
A group $G\subset GL_3(\mathbb R)$ is strictly generated in codimension two if and only if the cohomology of $U_G$ is concentrated in degree one and $\operatorname{h}^{1}(U_G)=2N-1$, where $N$ is the number of lines in the subspace arrangement $\mathcal A^G$.
\end{TheoM}
%

\section*{Questions}

In \Cref{pro-alternating-subgroup}, we prove that the alternating subgroup of a Coxeter group is strictly generated in codimension two. 
We wonder if any group strictly generated in real codimension two is the alternating group of a Coxeter group.

\begin{Question}\label{isAlt?}
Suppose $G$ is strictly generated in codimension two. 
Does there exist a Coxeter group $W$ such that $G=\alt(W)$?
\end{Question}

If $G=\alt(W)$, then $\mathcal A^G$ is the truncation of $\mathcal A^W$ at codimension $2$.
With this in mind, we ask a weaker version of \Cref{isAlt?}:

\begin{Question}\label{que:isSubArrangement}
Let $G$ be a group strictly generated in codimension two. 
Is there a reflection group $W$ such that $\mathcal{A}^G\subseteq \mathcal{A}^W$? 
\end{Question}

We have verified that \Cref{isAlt?}, hence also \Cref{que:isSubArrangement}, has an affirmative answer for subgroups of $GL_d(\mathbb R)$ for $d\leq 3$.








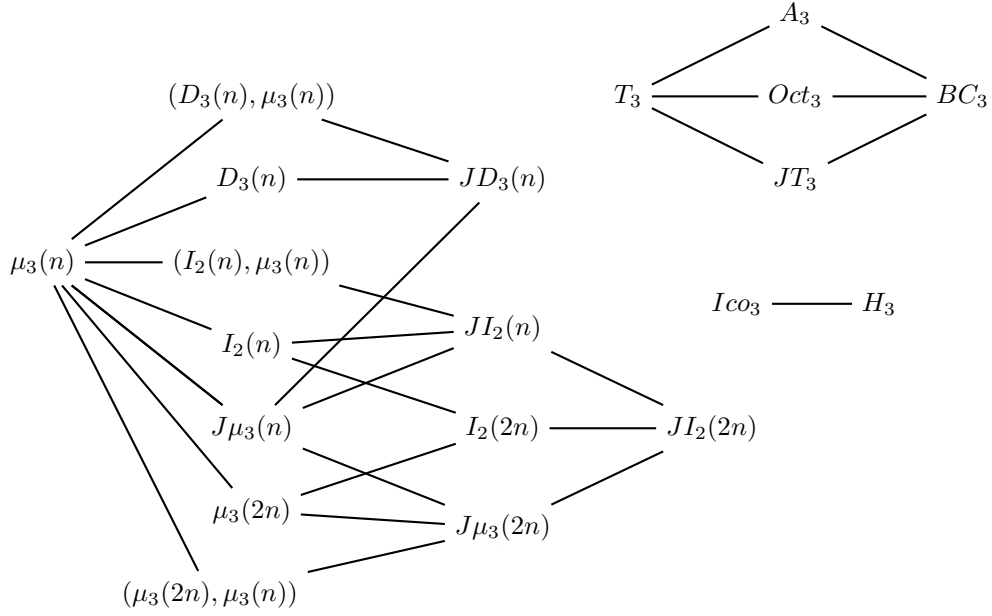
\begin{figure}[h]
\begin{tikzpicture}[thick,scale=1.10]
  \node (j-d3) at (5.5, 1) {$\J D_3(n)$};
  \node (d3) at (2.5, 1) {$D_3(n)$};
  \node (D3-mu2) at (2.5, 2) {$(D_3(n), \mu_3(n))$};

  \node (mu2) at (0,0) {$\mu_3(n)$};
  \node (d2-mu2) at (2.5, 0) {$(I_2(n), \mu_3(n))$};
  \node (j-mu2) at (2.5, -2) {$\J \mu_3(n)$};
  \node (d2) at (2.5, -1) {$I_2(n)$};
  \node (mu2-2n) at (2.5, -3) {$\mu_3(2n)$};
  \node (mu22n-mu2) at (2, -4) {$(\mu_3(2n),\mu_3(n))$};
  \node (j-d2) at (5.5, -0.8) {$\J I_2(n)$};
  \node (i2-2n) at (5.5,-2) {$I_2(2n)$};
  \node (j-mu2-2n) at (5.5, -3.2) {$\J \mu_3(2n)$};
  \node (j-d2-2n) at (8, -2) {$\J I_2(2n)$};

\draw (mu2) -- (j-mu2) -- (j-d3);

  \draw (j-mu2) -- (j-d2);
  \draw (d2) -- (i2-2n) -- (mu2-2n);
  \draw (i2-2n) -- (j-d2-2n);
  \draw (mu2) -- (d3);
  \draw (mu2) -- (D3-mu2);
  \draw (mu2) -- (d2);
  \draw (mu2) -- (d2-mu2);
  \draw (mu2) -- (mu22n-mu2);
  \draw (mu2) -- (j-mu2);
  \draw (mu2) -- (mu2-2n);
  
  \draw (d3) -- (j-d3);
  \draw (D3-mu2) -- (j-d3);
  
  \draw (d2) -- (j-d2);
  \draw (d2-mu2) -- (j-d2);
  
  \draw (j-mu2-2n) -- (j-mu2-2n);
  \draw (mu2-2n) -- (j-mu2-2n);
  
  \draw (j-mu2-2n) -- (j-d2-2n);
  \draw(j-d2) -- (j-d2-2n);
  \draw(j-mu2) -- (j-mu2-2n);
  \draw(mu22n-mu2) -- (j-mu2-2n);
  
  \node (T3) at (7, 2) {$T_3$};
  
  \node (j-T3) at (9, 1) {$\J T_3$};
  \node (Oct3) at (9, 2) {$Oct_3$};
  \node (A3) at (9, 3) {$A_3$};
  
  \node (BC3) at (11, 2) {$BC_3$};
  
  \draw (T3) -- (j-T3) -- (BC3);
  \draw (T3) -- (Oct3) -- (BC3);
  \draw (T3) -- (A3) -- (BC3);

  \node (Ico_3) at (8.3,-0.5) {$Ico_3$};
  \node (H3) at (10,-0.5) {$H_3$};

  \draw (Ico_3) -- (H3);

\end{tikzpicture} 
\caption{The figure shows the inclusion order of linear subgroups of $GL_3(\mathcal{R})$. In \Cref{sec:classification}, we discuss the inclusion of the corresponding subspace arrangements.}
\label{fig:all-arrangements}
\end{figure}

\vspace{0.2cm}

\begin{table}
\centering

\begin{tabular}{|L|L|L|L|L|}  
\hline
\text{Group}                &\text{Other Name}      &\mathcal R_1G      &\mathcal R_2G  &\text{codim}   \\  \hline\hline
\boldsymbol{\mu_3(n)}                    &\mu_2(n)\times\triv        &\triv[3]           &\boldsymbol{\mu_3(n)}       &2  \\  \hline
\J\mu_3(n),\text{ even }n   &\mu_2(n)\times A_1         &\triv[2]\times A_1 &\mu_3(n)       &2  \\  \hline
\J\mu_3(n),\text{ odd }n    &                           &\triv[3]           &\mu_3(n)       &\boldsymbol{3}  \\  \hline
(\mu_3(2n),\mu_3(n)),\text{ even }n &   &\triv[3]           &\mu_3(n)       &\boldsymbol{3}  \\  \hline 
(\mu_3(2n),\mu_3(n)),\text{ odd }n  &\mu_2(n)\times A_1  &\triv[2]\times A_1 &\mu_3(n)       &2  \\  \hline 
\boldsymbol{(D_3(n),\mu_3(n))}          &\boldsymbol{I_2(n)\times\triv}          &\boldsymbol{I_2(n)\times\triv}  &\mu_3(n)       &1  \\  \hline 
\boldsymbol{D_3(n)}  &                           &\triv[3]    &\boldsymbol{D_3(n)}         &2  \\  \hline 
\boldsymbol{\J D_3(n),} \textbf{ even }\boldsymbol{n}    &\boldsymbol{I_2(n)\times A_1}           &\boldsymbol{I_2(n)\times A_1}   &D_3(n)         &1  \\  \hline 

\J D_3(n),\text{ odd }n     &                           &I_2(n)\times\triv  &D_3(n)         &2  \\  \hline

(D_3(2n),D_3(n)),\text{ even }n     &                   &I_2(n)\times \triv                &D_3(n)         &2 \\  \hline 
\boldsymbol{(D_3(2n),D_3(n)),\textbf{ odd }n}      &\boldsymbol{I_2(n)\times A_1}  &\boldsymbol{I_2(n)\times A_1}   &D_3(n)         &1  \\  \hline 
\boldsymbol{T_3}     &                           &\triv[3]           &\boldsymbol{T_3} &2  \\  \hline 
\J T_3                      &                           &A_1\times A_1\times A_1&T_3        &2 \\  \hline 
\boldsymbol{(Oct_3,T_3)}                 &\boldsymbol{A_3}                        &\boldsymbol{A_3}                &T_3            &1  \\  \hline 
\boldsymbol{Oct_3}   &                           &\triv[3]      &\boldsymbol{Oct_3}          &2  \\  \hline 

\boldsymbol{\J Oct_3}                    &\boldsymbol{BC_3}                      &\boldsymbol{BC_3}              &Oct_3          &1  \\  \hline 
\boldsymbol{Ico_3}         &               &\triv[3]           &\boldsymbol{Ico_3}          &2  \\  \hline 
\boldsymbol{\J Ico_3}                   &\boldsymbol{H_3}                       &\boldsymbol{H_3}                &Ico_3          &1  \\  \hline 
\end{tabular}
\vspace{0.5cm}

\caption{Rank $3$ Linear Groups, rearranged in order of $\mathcal R_2G$. We specify with if the group is strictly generated in a certain codimension.}
\label{allGroupData}
\end{table}

\begin{acknowledgement}
    
    The first author has been partially supported 
    by the \emph{Zelevinsky Research Instructor Fund} and currently it is supported by the \emph{Knut and Alice Wallenberg Fundation} and by the \emph{Royal Swedish Academy of Science}.
\end{acknowledgement}


\setcounter{TheoM}{0}
\section{Preliminaries and generalities}\label{sec-preliminaries}

In this section we briefly review some basic terminology and results. 

\subsection{Posets}\label{sec:posets}
The mathematical objects we are going to study are partially ordered sets, called {\em posets}.
A standard reference for a complete and comprehensive introduction is \cite[Chapter 2]{Stanley}. 

A poset $(P,\leq)$ is a set $P$ together with a partial order relation $\leq$. In what follows the set $P$ will always be a finite set. 
If $p,q\in P$ and $p\leq q$, the {\em closed interval} is the subset $[p,q]\define \{r\in P \mid p\leq r\leq q\}$. 
A {\em lattice} is a poset where every pair of elements has a unique minimal upper bound as well as a unique maximal lower bound.
%

A \emph{chain} in $P$ is an ordered sequence $p_1<p_2<\ldots <p_k$ of elements in $P$.
We denote by $(\Delta (P),\subseteq)$ the poset of all chains of $P$, ordered by inclusion. 
If the poset $P$ has a unique minimal element $\hat 0$ and a unique maximal element $\hat 1$, we define the {\em reduced order complex of P} as 
  	\def\roc{\widehat{\Delta}} 
  	\[
  		\roc (P) \define  \Delta(P\setminus \{\hat 0, \hat 1\}),
  	\]
which is naturally a simplicial complex.


  

\subsection{Subspace Arrangements}
In this work, we study posets arising from {\em linear subspace arrangements}. 
Let $V$ be a real vector space of dimension $d$.
An {\em arrangement of linear subspaces} in $V$ is a finite collection $\mathcal{A}$ of linear subspaces of $V$. 
The poset of intersections associated to $\mathcal A$ is the set
\begin{equation*}
    \mathcal{L}(\mathcal{A})\define \left\{ \bigcap_{s\in S} s \mid S\subseteq \mathcal{A} \right\}
\end{equation*}
ordered by reverse inclusion: for $x,y\in \mathcal{L}(\mathcal{A})$, $x\leq y$ if $x\supseteq y$. 
The poset $\mathcal L(\mathcal A)$ is a lattice and it has a unique minimal element $\hat{0}=V$, that is the intersection of the empty family, and a unique maximal element $\hat{1}=\bigcap_{s\in \mathcal{A}} s$. (A subspace arrangement with a unique maximal element is called {\em central}.)
We say that the arrangement $\mathcal A$ is {\em essential} if $\dim\hat 1=0$.
%
%
%
%
%

  %


\subsection{The Subspace Arrangement of a Representation}
Let $G\xhookrightarrow{\rho} GL(V)$ be a finite dimensional real representation of $G$.
The normalizer $N_G(H)$ of a subgroup $H\subset G$ (resp. the stabilizer of a subset $S\subset V$) is denoted by \begin{align*}
    N_G(H)                      &\define\left\{g\in G\mid gH=Hg\right\}\\ 
    \operatorname{Stab}_\rho(S) &\define\{g\in G \mid gS\subseteq S\}
\end{align*}
If $v$ is a vector in $V$ and $S$ the singleton $\{v\}$, then we omit the set brackets by writing $\operatorname{Stab}_{\rho}(\{v\})$ as $\operatorname{Stab}_{\rho}(v)$. 
The latter is the isotropy group in $v$, that is the subgroup of $G$ that stabilizes the point $v$.

\noindent
Following \cite{Ekedahl-inv}, for any subgroup $H$ of $G$ we define:
\begin{align*}
    V^H_{\rho}  &\define \{v\in V\mid H\subseteq \operatorname{Stab}_{\rho}(v)\}\\
                &=\left\{v\in V\mmid Hv=v\right\}\subseteq V.
\end{align*}
In other words, $V^H_{\rho}$ is the subspace containing all the points fixed (at least) by all elements of $H$.

\begin{defi}
For $L$ a subspace of $V$, we denote by $\operatorname{Fix}_{\rho}(L)$ the \emph{fixator} of $L$: 
\[
	\operatorname{Fix}_{\rho}(L)\define \{g\in G \mid\forall\,l\in L,\,gl = l\}\subseteq G.
\]
\end{defi}

\begin{defi}
A subgroup $H\subset G$ is called a \emph{stabilizer subgroup} with respect to the representation $G\xhookrightarrow{\rho} GL(V)$ if $\operatorname{Fix}_{\rho}(V_\rho^H)=H$.
Equivalently, $H$ is the largest subgroup of $G$ fixing $V_\rho^H$ point-wise.
\end{defi}

If $H$ is a stabilizer subgroup with respect to the representation $\rho$, then $\operatorname{Fix}_{\rho}(L)\subseteq \operatorname{Stab}_{\rho}(L)$ and $\operatorname{Stab}_{\rho}(V^H)=N_G(H)$.

\begin{defi}\label{defi:subspace-arrangement}
The subspace arrangement $\mathcal A^\rho$ corresponding to the representation $\rho$ is
\[
\mathcal{A}^{\rho}  \define\left\{V^{H}_{\rho} \mmid \{e\}\neq H \mbox{ is a stabilizer subgroup with respect to } \rho\right\}.
\]
We also denote set of codimention-$i$ subspaces by
\[
	\mathcal{A}_i^{\rho}\define\left\{V^{H}_{\rho}\in\mathcal A^\rho\mmid \operatorname{codim} V^H_\rho=i \right\}.
\]
Moreover we denote by $U_G$ the open complement of $\mathcal A^\rho$ in the vector space $V$.
\end{defi}
\noindent
A complete description of such arrangements can be found in \cite{DelucchiMartinoBW}. 

It is a simple observation (see \cite[Section 1.2]{DelucchiMartinoBW}) that the lattice of intersection $\mathcal L(\mathcal A^\rho)$ of $\mathcal{A}^{\rho}$ is precisely $\mathcal{A}^{\rho}$ together with $V=V^e$; thus 
\begin{equation}\label{eq:latticeOfIntersection}
\mathcal{L}(\mathcal{A}^{\rho})=\{V^{H}_{\rho} \mid H \mbox{ is a stabilizer subgroup of } G \mbox{ w.r.t. } \rho\}.
\end{equation}

\noindent
The atoms of such lattice are precisely the maximal subspace by inclusion.

The reduced order complex of $\mathcal{L}(\mathcal{A}^{\rho})$ has important geometric and group theoretic properties, see \cite[Section 3]{DelucchiMartinoBW,Ekedahl-inv}. 
We are going to show in \Cref{sec:cohomology-computations} that $\roc \mathcal{L}(\mathcal{A}^{\rho})$ has a crucial role in the computation of the cohomology of $U_G$.

\section{Finite groups generated in high codimension}
In this section we describe the class of finite groups that we study in this work.
These groups are a natural generalization of finite reflection groups.

\vspace{0.1cm}

All along this section, $G$ is a finite linear group $G\xhookrightarrow{\rho} GL(V)$. We denote by $|G|$ the underlying abstract group of $G$.

\vspace{0.1cm}

Let us set some useful notations. The \emph{rank} $rk(G)$ of $G$ is defined to be the dimension of $V$.
%
%
%
Given two linear groups $G_1\xhookrightarrow{\rho_1} GL(V_1)$ and $G_2\xhookrightarrow{\rho_2} GL(V_2)$ we define the direct product $G_1\times G_2$ as a linear group in $GL(V_1\times V_2)$ with representation $\rho_1\times \rho_2$.
We denote by $\triv[n]$ the trivial group $\{e\}$, viewed as a linear group in $GL_n(\mathbb R)$, i.e.,
\[
	\triv[n]:=\left\{\begin{bmatrix}
    1 & &  &   &  \\
      & 1 &  &  &  \\
      &  & \ddots &  &  \\
      &  & &  & 1
\end{bmatrix} \right\}\subset GL_n(\mathbb R).
\]
Observe that $rk(\triv[n])=n$, and $rk(H\times G)=rk(H)+rk(G)$.




\begin{defi}
We say that $G$ is \emph{essential} if the associated subspace arrangement $\mathcal A^\rho$ is essential.
\end{defi}
It is clear that $\triv[n]$ is non-essential.
Further, if a non-trivial group $G$ is non-essential, there exists an essential subgroup $H\subset G$ such that $G=H\times \triv[d]$ for some $d$, and $|G|\cong|H|$.

\begin{lem}
\label{lem:equivariant}
Let $G$ be a finite linear group.
Then there exists a $|G|$-equivariant positive definite bilinear form $\omega$ on $V$.
In particular, we may assume $G\subset O(V)$.
\end{lem}
\begin{proof}
The group $G$ acts on the set of bilinear forms on $V$ via:\[
    (g\cdot\theta)(v,w)=\theta(g^{-1}v,g^{-1}w)
\]
Fix some positive definite bilinear form $\theta$ on $V$.
Then $$\omega=\sum\limits_{g\in G}g\cdot\theta$$ is a $G$-equivariant positive definite bilinear form on $V$.
\end{proof}

\begin{defi}
An element $r\in GL(V)$ is called a {\em reflection} if it has finite order and it fixes a subspace of codimension one, called the {\em reflection hyperplane}.
Following the notations given in \Cref{sec-preliminaries}, the reflection hyperplane is denoted $V^{\langle r\rangle}$.
\end{defi}

\begin{defi}
An element $g\in GL(V)$ is called a {\em rotation} if it has finite order, and further $\operatorname{codim}V^{\langle g\rangle}=2$.
\end{defi}

\begin{defi}
A group $G\subseteq GL(V)$ is a reflection group if it is generated by reflections, i.e.,
$    G=\left\langle g\mmid\operatorname{codim}{V^{\langle g\rangle}}=1\right\rangle$.
\end{defi}

\begin{defi}\label{def-genetated-in-codimention-k}
We say that a finite linear group $G$ in ${GL}(V)$ is generated in codimension $k$ if 
$G$ is generated by elements of codimension at most $k$, i.e., $G=\left\langle g\mmid \operatorname{codim}{V^{\langle g\rangle}}\leq k\right\rangle$.
\end{defi}

We want to emphasize that being a \emph{reflection} is a property of the representation $\rho$, i.e., of the linear group $G$ and not of the abstract group $|G|$ itself. 
For instance, the group $\nicefrac{\mathbb Z}{2\mathbb Z}$, viewed as the subgroup  of $GL_1(\mathbb{R})$ generated by $-1$, is a reflection group; on the contrary, $\nicefrac{\mathbb{Z}}{2\mathbb{Z}}$ is not a reflection group if we embed it in $GL_2(\mathbb{R})$ with generator $-\operatorname{Id}$, the linear antipodal map. 
In the latter case $\nicefrac{\mathbb{Z}}{2\mathbb{Z}}$ is generated in codimension two.
Other interesting examples of this phenomena are discussed later; nevertheless we suggest the reader have a look at \Cref{sec:mixed-d3n-mu3n,sec:D3n-in-GL-3}.

For our goal, we want to distinguish between abstract groups, for example $\nicefrac{\mathbb{Z}}{2\mathbb{Z}}$, which admit a representation as a linear group generated in low codimension, and abstract groups, like 
$\nicefrac{\mathbb Z}{n\mathbb Z}$ that \emph{only} admit representations in higher codimensions.

\begin{defi}
We say that $G$ is \emph{strictly} generated in codimension $k$ if $G=\left\langle g\mmid \operatorname{codim}V^{\langle g\rangle} = k\right\rangle$.
\end{defi}

\begin{pro}\label{Pro-gen-codim-2-in-high-dimension}
Let $G_1,\dots, G_r$ be finite linear groups (strictly) generated in codimension $k$.
Then the product group $G_1\times \dots \times G_n$    is (strictly) generated in  codimension $k$.
\end{pro}
\begin{proof}
Let $G_1,\dots, G_r$ be finite linear groups of rank $j_1,\dots,j_r$ respectively (strictly) generated in codimension $k$, and let $V_i$ denote the vector space underlying the linear group $G_i$. 
We identify the product group $G=G_1\times\dots\times G_r$ as a subgroup of $GL(V)$, where $V=V_1\times\dots\times V_r$.
Observe that\[
    \operatorname{codim}_{V_i}V_i^{\langle g_i\rangle}=\operatorname{codim}_VV^{\langle(1,\dots,g_i,\dots,1)\rangle}
\]
It follows that the subgroup $H_i=\triv[j_1]\times\dots\times G_i\times\dots\triv[j_r]$ is (strictly) generated in codimension $i$.
Since $G$ is generated by its subgroups $H_i$, we deduce that $G$ is also (strictly) generated in codimension $i$.
\end{proof}

\begin{defi}
An abstract group $|G|$ is \emph{minimally generated} in codimension $k$ if $k$ is the minimal integer for which there exists a faithful representation $\rho:|G|\rightarrow GL(V)$ such that $\rho(|G|)$ is strictly generated in codimension $k$.
\end{defi}


For instance, the alternating subgroup $|T_3|$ of the symmetric group on four elements is minimally generated in codimension two.
In \Cref{sec:T_3}, we have shown that $|T_3|$ is strictly generated in codimension two; moreover by the classification of finite reflection groups of rank three \cite{MR1890629}, we know that the only reflection group of order $12$ is $I_2(6)$.
One simply verifies that $|I_2(6)|\not\cong|T_3|$ (for instance, $|I_2(6)|$ has a normal subgroup of order $6$, while $|T_3|$ does not). 
See also \Cref{sec:ReflectionSmallRank}.

\subsection{Strictly Generated Subgroups}
Following \Cref{defi:subspace-arrangement}, we set $\mathcal A^G_i  \define\left\{W\in\mathcal A^G\mmid\operatorname{codim} W=i\right\}$. 
We define the subgroups $\mathcal R_i G$,
\begin{align*}
    \mathcal R_i G  &\define\left\langle g\in G\mmid\dim V^{\langle g\rangle}=i\right\rangle.
\end{align*}
In particular, $\mathcal R_1 G$ (resp. $\mathcal R_2 G$) is the subgroup of $G$ generated by the set of reflections (resp. rotations).
Observe that for $g,h\in G$, we have $V^{\langle hgh^{-1}\rangle}=hV^{\langle g\rangle}$, hence $\dim V^{\langle g\rangle}=\dim V^{\langle hgh^{-1}\rangle}$.
It follows that the $\mathcal R_i G$ are normal subgroups of $G$.

\subsection{Reflection Groups of Small Rank}\label{sec:ReflectionSmallRank}
Linear reflection groups were classified up to conjugacy by Coxeter \cite{MR1503182}.
We list here all the linear reflection groups of $GL_d(\mathbb R)$ for $d=1,2,3$, following the notation of \cite{MR1890629}.

The linear group $A_1=\{\pm 1\}\subset GL_1(\mathbb R)$, generated by reflection across the origin $O$ is the unique non-trivial reflection group in rank $1$. 

Consider the set $P_n$ containing $n$ equi-distant points on the circle of radius one.
The linear group $I_2(n)$, for $n\geq 1$, is defined to be the stabilizer subgroup of $P_n$ in $O_2(\mathbb R)$, i.e.,
\[ 
    I_2(n)=\left\{g\in O_2(\mathbb R)\mmid gP_n=P_n\right\}.
\]
Observe that $I_2(1)=A_1\times\triv$ and $I_2(2)=A_1\times A_1$.
Following convention, we set $A_2=I_2(3)$ and $BC_2=I_2(4)$.
The groups $I_2(n)$, for $n\geq 2$, are precisely all the essential groups of rank $2$.
The non-essential ones are $I_2(1)$ and $\triv[2]$.





\section{Groups of Rank Two}\label{sec:rank-two}
Consider a non-trivial group $G$ of rank $1$.
Any element $g\in G$ either fixes the whole ambient space $\mathbb{R}$, or is a reflection. 
We deduce that $G=A_1$. 
Therefore, we start by studying groups of rank two. 
We use the classification of linear Coxeter groups by \emph{Coxeter diagrams} \cite{MR1890629}; we have set our notations in \Cref{sec:ReflectionSmallRank}.




\begin{pro}\label{prop-GL2-real}
Let $G\subset GL_2(\mathbb R)$ be a finite linear group which is not generated in codimension $1$.
Then $\subweyl G=\triv[2]$, and further\[
    G=\mu_2(n)\define\mathcal R_2I_2(n),
\]
for some $n\geq 2$.
The arrangement $\mathcal{A}^{G}$ contains only the origin and the full space $V=\mathbb R^2$, i.e., $\mathcal{A}^{G}=\{O < V\}$. 
\end{pro}
\begin{proof}
Observe that $|\mu_2(n)|=\nicefrac{\mathbb Z}{n\mathbb Z}$.
If $\subweyl G=\triv[2]$, then $G$ is the cyclic group generated by a single rotation $\theta$.
Consequently, we have $G=\mathcal R_2I_2(n)$, where $n$ is the order of $\theta$.

It remains to show that $\subweyl G=\triv[2]$.
Suppose this is false.
By \Cref{sec:ReflectionSmallRank}, we have $G=I_2(n)$ for some $n\geq 1$.
%
%
Then $\mathcal A_1^G=\mathcal A_1^{\subweyl G}$ contains $n$ lines, with consecutive lines being separated by angle $\nicefrac\pi n$.
Since $G$ preserves both the circle of radius one (since $G\subset O_2(\mathbb R)$), and the sub-arrangement $\mathcal A_1^G$, it also preserves their intersection $P_{2n}$. 
Now, the group of symmetries of $P_{2n}$ being $I_2(2n)$, we have $I_2(n)\subsetneq G\subset I_2(2n)$.
It follows that $G=I_2(2n)$, since $I_2(n)$ has index $2$ in $I_2(2n)$, hence is a maximal proper subgroup. This contradicts the assumption that $G=I_2(n)$.
%
\end{proof}

For $n>2$, $|\mu_2(n)|\cong \nicefrac{\mathbb Z}{n\mathbb Z}$ is not a Coxeter group; hence, $\nicefrac{\mathbb Z}{n\mathbb Z}$ has no faithful real reflection representation.
On the other hand, the previous proposition shows that there exists a faithful representation for which $\nicefrac{\mathbb Z}{n\mathbb Z}$ is strictly generated in codimension two.
Therefore this is the first example of a finite group that is minimally generated in real codimension two.

\begin{Cor}\label{Cor-mu-2-n>2isNeatInCod2}
For $n>2$, the abstract group $\nicefrac{\mathbb Z}{n\mathbb Z}$ is minimally generated in codimension two. 
\end{Cor}

	
	

\section{Groups of Higher Rank}\label{sec:higherRealDimension}

In this section we study the subspace arrangement of a group of higher rank. 
We start with a simple observation.

\begin{lem}\label{lemma-rotation-and-reflections}
Let $g$ be an finite order element of ${GL}_n(\mathbb{R})$. 
Then $g$ is a product of reflections and rotations.
\end{lem}
\begin{proof}
Since $\langle g\rangle$ is finite, we may assume $g\in O_n(\mathbb R)$.
It follows that $g$ is conjugate to some block matrix:\[
    g=\begin{pmatrix}R_1&&&&\\
                        &\ddots&&&\\
                        &&R_d&&\\
                        &&&-I_k&\\
                        &&&&I_l
    \end{pmatrix}
\]
where $I_k, I_l$ are the identity matrix of size $k,l$ resp., and \[
    R_i =\begin{pmatrix}\cos\theta_i    &\sin\theta_i\\
                        -\sin\theta_i   &\cos\theta_i
        \end{pmatrix}
\]
We can now write $g$ as product of matrices of the form
\begin{align*}\begin{pmatrix}I_a&&\\
                &-1&\\
                &&I_b
\end{pmatrix}&&\text{and}\qquad\qquad&
\begin{pmatrix}I_{a'}&&\\
                &R_i&\\
                &&I_{b'}
\end{pmatrix}
\end{align*}
which correspond to reflections and rotations respectively. (Observe that $a+b=n-1$ and $a'+b'=n-2$.)
\end{proof}

We note that the arrangement of a group generated in codimension $i>2$ might contain subspaces in lower codimension.
Take for instance the group $G$ generated by the ten diagonal matrices having in the diagonal only two $1$'s and three $-1$'s. 
This group is generated in codimension three, but 
\[
 \begin{bmatrix}
    -1 &  &  &  &  \\
      & -1&  &  & \\
      &  &-1 &  &  \\
      &  &  & 1 &  \\
      &  &  &   & 1  
\end{bmatrix}
\begin{bmatrix}
    1 &  &  &  &  \\
      & -1&  &  & \\
      &  &-1 &  &  \\
      &  &  & -1 &  \\
      &  &  &   & 1  
\end{bmatrix}
=
\begin{bmatrix}
    -1 &  &  &  &  \\
      & 1&  &  & \\
      &  &1 &  &  \\
      &  &  & -1 &  \\
      &  &  &   & 1  
\end{bmatrix}
\]
is a rotation and therefore the subspace arrangement contains a codimension two subspace.

One of the perks of groups strictly generated in codimension two is that there are no reflection hyperplanes in the corresponding subspace arrangement.

\begin{pro}\label{pro:no-cod-one}
If $G\subseteq GL_n(\mathbb{R})$ and $G$ is strictly generated in codimension two, then 
$\mathcal{A}^G_1=\emptyset$.
\end{pro}
\begin{proof}
By \Cref{lem:equivariant}, we may assume $G\subseteq O_n(\mathbb{R})$.
Now, since $G$ is strictly generated in codimension two, it follows from \Cref{flgso3} that $G\subseteq SO_n(\mathbb{R})$.
Suppose that there exists a hyperplane  $V^H$ in $\mathcal{A}^G$, where $H$ is a stabilizer subgroup of $G$.
Any non-trivial element $r\in H$ is a reflection, hence satisfies $\det(r)=-1$.
This contradicts the assumption that $H\subset SO_3(\mathbb R)$.
\end{proof}

Alternating subgroups of Coxeter groups are strictly generated in codimension two.

\begin{pro}\label{pro-alternating-subgroup}
For $W$ a Coxeter group, we have $\mathcal R_2W=\alt(W)$.
\end{pro}
\begin{proof}
Recall that the alternating subgroup $\alt(W)$ is the kernel of the determinant map $\det:W\rightarrow\mathbb R$.
In particular, $\mathcal R_2W\subset SO_n(\mathbb R)$ implies that $\mathcal R_2W\subset\alt(W)$. 

\noindent
Conversely, any element of $\alt(W)$ is the product of an even number of reflections.
Since the product of two reflections is a rotation, it follows that $\alt(W)$ is generated by rotations, hence $\alt(W)\subset\mathcal R_2W$.
\end{proof}

As we highlight in the introduction, we wonder in \Cref{isAlt?} if the class of alternating subgroups of Coxeter groups coincides with the class of groups generated in codimension two.
Right after \Cref{pro:subgroup-O3} we have verified that this is the case in $GL_3(\mathbb{R})$.


Recall that the lattice of intersection $\mathcal L(\mathcal{A}^G)$ of $\mathcal{A}^G$ is the poset $\mathcal{A}^G$ together with a minimal element $\hat{0}$ corresponding to $V^e$, see \cref{eq:latticeOfIntersection}.
A lattice $L$ is atomic if every $x$ in $L$ is the join of some atoms. It is a classical result that the lattice of intersections of a reflection arrangement is geometric, hence also atomic.


\begin{Theo}\label{thm-atomic-lattice}
Let $G$ be an essential group generated in codimension two. 
Then the lattice of intersection $\mathcal{L}(\mathcal{A}^G)$ is atomic.
\end{Theo}
\begin{proof}
  Let $x$ be in $\mathcal{L}(\mathcal{A}^{\rho})$ and consider the closed interval $I=[\hat{0}, x]$. If the atoms of $I$ are only hyperplanes, $h_1,\dots, h_k$, than there is nothing to prove, because as we said above the lattice of intersection of a reflection hyperplanes is a geometric lattice and $x$ is the joins of $h_1,\dots, h_k$.
  
  Assume now that some of the atoms of $I$ are subspaces $S_1, \dots, S_h$ of codimension two. Consider the sub-lattice $\mathcal{B}$ of $I$ defined by all joins of reflection hyperplanes in $I$.
  
  Since $x\notin \mathcal{B}$, we call $y$ the maximal element in $\mathcal{B}$ and $y$ properly contains $x$. 
  In other words, $y$ corresponds geometrically to the intersection of all hyperplanes in $I$ and, since, $x$ is not $\mathcal{B}$ then $y$ is properly contained in $x$
  Then, $x$ is the join of such $y$ with the remaining atoms $S_1, \dots, S_h$ in $I$.
\end{proof}

There is no hope that these arrangements are geometric lattices; a quick glance at \Cref{sec:classification} reveals several pathological examples, see for example \Cref{tab:fig:arr-D_oddxI}.


\section{Groups of Rank Three}\label{sec:rank-three}
In this section and in \Cref{sec:classification}, we study the finite subgroups of $GL_3(\mathbb R)$.
In particular, we identify the subgroups $G\subset GL_3(\mathbb R)$ for which the abstract group $|G|$ is minimally generated in codimension two.

We remark that by \Cref{lem:equivariant}, we may assume $G\subset O_3(\mathbb R)$; moreover, if $G$ is strictly generated in codimension two, then $G\subset SO_3(\mathbb R)$.

\begin{lem}\label{flgso3}
A subgroup $G\subset O_3(\mathbb R)$ is strictly generated in codimension two if and only if $G\subset SO_3(\mathbb{R})$.
\end{lem}
\begin{proof}
Any rotation $g\in O_3(\mathbb R)$ satisfies $\det g=1$, i.e., $g\in SO_3(\mathbb R)$.
It follows that $G\subset SO_3(\mathbb R)$.
Conversely, since every non-identity element of $SO_3(\mathbb R)$ is a reflection, any subgroup $G\subset SO_3(\mathbb R)$ is strictly generated in codimension two.
\end{proof}

In this section, we prove some theoretical results.
In \Cref{sec:classification}, we compute the subspace arrangements for all $G\subset O_3(\mathbb R)$. 
We occasionally refer to some examples presented in the final section. 

\begin{lem}\label{lem:bou-ref}
Following \cite{MR1890629}, the following are the essential rank three linear reflection groups:\begin{enumerate}
    \item The product group $A_1\times I_2(n)$. 
    \item The group of symmetries of a regular tetrahedron, denoted $A_3$.
    \item The group of symmetries of a cube, denoted $BC_3$.
    \item The group of symmetries of a regular icosahedron, denoted $H_3$.
\end{enumerate}
\end{lem}

\begin{pro}[cf. \cite{Klein-icosahedron}]\label{pro:Klein-class}
A (finite) subgroup $G\subset O_3(\mathbb R)$, strictly generated in codimension $2$ is, up to conjugation, one of the following:\begin{enumerate} 
    \item The Cyclic group $\mu_2(n)\times\triv$.
    \item The Dihedral group $D_3(n)$; it is the group of automorphisms in $SO_3(\mathbb R)$ of a regular $n$-gon.
    \item The Tetrahedral group $T_3$; it is the group of automorphisms in $SO_3(\mathbb R)$ of the regular tetrahedron.
    \item The Octahedral group $Oct_3$; it is the group of automorphisms in $SO_3(\mathbb R)$ of the cube.
    \item The Icosahedral group $Ico_3$; it is the group of automorphisms in $SO_3(\mathbb R)$ of the regular icosahedron.
\end{enumerate}
For convenience, we write $\mu_3(n)\define\mu_2(n)\times\triv$.
\end{pro}



Observe that for each group $G$ in the previous proposition, there exists a reflection group $W$ such that $G=\mathcal R_2W$.
Indeed, we can verify \begin{align*}
    \mu_3(n)=\mathcal R_2(I_2(n)\times\triv),\qquad D_3(&n)=\mathcal R_2(I_2(n)\times A_1),\\
    T_3=\mathcal R_2A_3,\qquad Oct_3=\mathcal R_2BC_3&,\qquad Ico_3=\mathcal R_2H_3. 
\end{align*}
This settles \Cref{isAlt?} for rank three groups.

We can use the previous proposition along with \Cref{Pro-gen-codim-2-in-high-dimension} to construct linear groups of higher rank which strictly generated in codimension two.
Indeed, if $G_1,\dots, G_n$ are any groups in \Cref{pro:Klein-class}, the product group $G_1\times\dots\times G_n$ is a linear group of rank $3n$, strictly generated in codimension two.

Moreover, by tensoring the $G$-representation $\rho$, we obtain higher rank linear groups strictly generated in higher codimension.

\begin{pro}\label{Pro-group-gen-in-codim-2n}
Let $G=(|G|,\rho)$ be a rank three linear group strictly generated in codimension two, i.e., one of the linear groups in \Cref{pro:Klein-class}.
The $n^{th}$-tensor $\operatorname{diag}_n G$, given by the faithful representation $$|G|\xhookrightarrow{\rho \times \dots \times \rho} GL(V^n)$$
is a rank $3n$ linear group strictly generated in codimension $2n$.
\end{pro}
\begin{proof}
It is sufficient to observe that $
    \operatorname{codim}(V^n)^{\langle (g,\dots,g)\rangle}=n\operatorname{codim} V^{\langle g\rangle}
.$
\end{proof}

Note that the abstract group underlying $\operatorname{diag}_n G$ is $|G|$, the same as the abstract group underlying $G$; but $\operatorname{diag}_n G$ is, of course, different from $G$ as a linear group.

We can now identify the abstract groups underlying the linear groups in \Cref{pro:Klein-class} which are minimally generated in codimension two.
\begin{TheoM}\label{THM-minimal-rank-3}
Let $G$ be a rank three linear group.
Then $|G|$ is minimally generated in codimension two precisely when $G$ is one of the following:
the Cyclic group $\mu_3(n)$ for $n\geq 3$, the Tetrahedral group $T_3$, or the Icosahedral group $Ico_3$.
\end{TheoM}
The proof of \Cref{THM-minimal-rank-3} is by inspection of all groups strictly generated in codimension two in \Cref{pro:Klein-class} compared with the list of all reflection groups, classified in \cite{MR1890629}.

The rest of this section is devoted to classifying the finite linear groups of rank three.

\subsection{Mixed Groups}\label{def:mixedgroup}
Let $\J:V\rightarrow V$ denote the \emph{antipodal map} $v\mapsto -v$.
For $G$ a linear group, we denote by $\J G$ the linear group generated by $G$ and \J.
Let $H,K$ be subgroups of $SO_3(\mathbb R)$ such that the index of $H$ in $K$ is $2$.
We define the mixed group \[
    (K,H)\define H\sqcup\left\{\J g\mmid g\in K\backslash H\right\}.
\]
Observe that there exists an isomorphism $|(K,H)|\cong|K|$, given by \[
    g\mapsto\begin{cases} 
                g   &\text{if }g\in H,\\ 
                \J g&\text{otherwise.}
            \end{cases}
\]

\begin{lem}
Let $G$ denote the mixed group $(K,H)$.
Then $\mathcal R_2G=H$.
\end{lem}
\begin{proof}
It is clear that $H\subset\mathcal R_2G$.
Further, since $G\not\subset SO_3(\mathbb R)$, we have $\mathcal R_2G\neq G$, i.e., $\mathcal R_2G$ is a proper subgroup of $G$.
Since $H$ is a maximal proper subgroup of $G$, it follows that $\mathcal R_2G=H$.
\end{proof}

\begin{pro}
Let $G$ be a finite subgroup of $O_3(\mathbb R)$. 
Exactly one of the following holds:\begin{enumerate}
    \item $G\subset SO_3(\mathbb R)$.
    \item $G=\J H$ for some $H\subset SO_3(\mathbb R)$.
    \item $G$ is a mixed group.
\end{enumerate}
\end{pro}
\begin{proof}
Let $G$ be a finite subgroup of $O_3(\mathbb R)$ not contained in $SO_3(\mathbb R)$.
Consider the map $\phi:G\rightarrow SO_3(\mathbb R)$ given by\[
    \phi(A)=\begin{cases}
                A   &\text{if }\det A=1,\\
                \J A&\text{otherwise.}
\end{cases}
\]
We set $H=\ker\phi$.
If $\J\in G$, then $G=\J H$.
Suppose now that $\J\not\in G$.
Then for any $g\in G\backslash H$, we have $\J g\not\in H$.
Consider the subgroup\[
    K=H\sqcup\left\{\J g\mmid g\in G\backslash H\right\}.
\]
It is straightforward to verify that $G=(K,H)$.
\end{proof}

\begin{pro}\label{pro:subgroup-O3}
The finite subgroups of $O_3(\mathbb R)$ are, up to conjugation, precisely the following:\begin{enumerate}
    \item The finite subgroups of $SO_3(\mathbb R)$.
    \item $\J H$ where $H$ is some linear group generated in codimension $2$.
    \item The mixed group $(Oct_3,T_3)=A_3$.
    \item The mixed group $(\mu_3(2n),\mu_3(n))$. 
    \item The mixed group $(D_3(n),\mu_3(n))=I_2(n)\times\triv$. 
    \item The mixed group $(D_3(2n),D_3(n))$.
\end{enumerate}
\end{pro}

We show in \Cref{sec:classification} that the reflection groups mentioned in \Cref{lem:bou-ref} occur in the above list as follows: $A_3=(Oct_3,T_3)$, $BC_3=\J Oct_3$, $H_3=\J Ico_3$, $I_2(n)\times\triv=(D_3(n),\mu_3(n))$, and
\begin{align*}
    I_2(n)\times A_1&=\begin{cases}\J D_3(n) &\text{if $n$ is even,}\\
                        (D_3(2n),D_3(n))    &\text{if $n$ is odd.}
                    \end{cases}
\end{align*}

\newpage
\def\percentage{0.71}
\begin{table}[h]
\begin{tabular}{ccc}
G &$\mathcal{A}^G$ & $\nicefrac{\mathcal A^G}G$ \\ \hline

$\mu_3(n) $   &
\begin{tikzpicture}[thick, scale=\percentage]
  \node (zero) at (0,2) {$l_0$};
  \node (whole) at (0,0) {$\mathbb R^3$};
  \draw (whole) -- (zero);
\end{tikzpicture} &
\begin{tikzpicture}[thick, scale=\percentage]
  \node (zero) at (0,2) {$\nicefrac{l_0}{G}$};
  \node (whole) at (0,0) {$\nicefrac{\mathbb R^3}{G}$};
  \draw (whole) -- (zero);
\end{tikzpicture}\\ \hline

\begin{tabular}{c}
$\J\mu_3(n)$\\
\mbox{for even } $n$ 
\end{tabular}
&
\begin{tikzpicture}[thick, scale=\percentage]
  \node (zero) at (0,2) {$0$};
  \node (l0) at (-1,0){$l_0$};
  \node (pi) at (1,-1){$l_0^\perp$};
  \node (whole) at (0,-3){$\mathbb R^3$};
  \draw (whole) -- (l0)--(zero);
    \draw (whole) -- (pi)--(zero);
\end{tikzpicture} &
\begin{tikzpicture}[thick, scale=\percentage]
  \node (zero) at (0,2) {$\nicefrac 0G$};
  \node (l0) at (-1,0){$\nicefrac{l_0}{G}$};
  \node (pi) at (1,-1){$\nicefrac{l_0^\perp}{G}$};
  \node (whole) at (0,-3){$\nicefrac{\mathbb R^3}G$};
  \draw (whole) -- (l0)--(zero);
    \draw (whole) -- (pi)--(zero);
\end{tikzpicture} \\ \hline

\begin{tabular}{c}
$\J\mu_3(n)$\\
\mbox{for odd } $n$ 
\end{tabular}
&
\begin{tikzpicture}[thick, scale=\percentage]
  \node (zero) at (0,2) {$O$};
  \node (l0) at (0,0) {$l_0$};
  \node (whole) at (0,-2) {$\mathbb R^3$};
  \draw (whole) -- (l0)--  (zero);
\end{tikzpicture} &
\begin{tikzpicture}[thick, scale=\percentage]
  \node (zero) at (0,2) {$\nicefrac{O}G$};
  \node (l0) at (0,0) {$\nicefrac{l_0}G$};
  \node (whole) at (0,-2) {$\nicefrac{\mathbb R^3}G$};
  \draw (whole) -- (l0) -- (zero);
\end{tikzpicture}\\ \hline

\begin{tabular}{c}
$(\mu_2(2n),\mu_2(n))$\\
\mbox{for even } $n$ 
\end{tabular}
&
\begin{tikzpicture}[thick,scale=\percentage]
    \node (zero) at (0,3.5) {$0$};
    \node (l) at (0,1.8) {$l_0$};
    \node (whole) at (0,0) {$\mathbb R^3$};
    \draw (zero) -- (l) -- (whole);
\end{tikzpicture}&
\begin{tikzpicture}[thick,scale=\percentage]
    \node (zero) at (0,3.5) {$0$};
    \node (l) at (0,1.8) {$\nicefrac{l_0}G$};
    \node (whole) at (0,0) {$\nicefrac{\mathbb R^3}G$};
    \draw (whole) -- (l) -- (zero);
\end{tikzpicture}\\ \hline

\begin{tabular}{c}
$(\mu_2(2n),\mu_2(n))$\\
\mbox{for odd } $n$ 
\end{tabular}
&
\begin{tikzpicture}[thick, scale=\percentage]
    \node (zero) at (0,4) {$O$};
    \node (l) at (1,2) {$l_0$};
    \node (lp) at (-1,0.5) {$l_0^\perp$};
    \node (whole) at (0,-1.5) {$\mathbb R^3$};
    \draw (whole) -- (lp) -- (zero) -- (l) -- (whole);
\end{tikzpicture}&
\begin{tikzpicture}[thick, scale=\percentage]
    \node (zero) at (0,4) {$O$};
    \node (l) at (1,2) {$\nicefrac{l_0}G$};
    \node (lp) at (-1,0.5) {$\nicefrac{l_0^\perp}G$};
    \node (whole) at (0,-1.5) {$\nicefrac{\mathbb R^3}G$};
    \draw (whole) -- (l) -- (zero) -- (lp) -- (whole);
\end{tikzpicture}\\ \hline

\end{tabular} 

\caption{The subspace arrangements of $\mu_3(n)$, $\J \mu_3(n)$, and $(\mu_2(2n), \mu_2(n))$.}
\label{tab:fig:arr-C_n}
\label{tab:fig:arr-C_evenxI}
\label{tab:fig:arr-C_oddxI}
\label{c2ncne}
\label{c2ncno}
\end{table}
\newpage

\def\percentage{0.75}
\begin{table}[h]
\begin{tabular}{ccc}
G &$\mathcal{A}^G$ & $\nicefrac{\mathcal A^G}G$ \\ \hline

\begin{tabular}{c}
$(D_3(n),\mu_3(n) )$\\
\mbox{for odd } $n$ 
\end{tabular}
&
\begin{tikzpicture}[thick,scale=\percentage]
    \node (l0) at (0,2) {$l_0$};
    \node (l1p) at (-1.5,0) {$\pi_1$};
    \node (lip) at (0,0) {$\dots$};
    \node (lnp) at (1.5,0) {$\pi_n$};
    \node (whole) at (0,-2) {$\mathbb R^3$};
   \draw  (l0) -- (l1p) -- (whole) -- (lip) -- (l0) -- (lnp) -- (whole);
\end{tikzpicture}
&
\begin{tikzpicture}[thick,scale=\percentage]
    \node (l) at (0,2) {$\nicefrac{l_0}G$};
    \node (lp) at (0,0) {$\nicefrac{\pi_0}G$};
    \node (whole) at (0,-2) {$\nicefrac{\mathbb R^3}G$};
    \draw (whole) -- (lp) -- (l);
\end{tikzpicture}\\ \hline

\begin{tabular}{c}
$(D_3(n),\mu_3(n) )$\\
\mbox{for even } $n$ 
\end{tabular}
&
\begin{tikzpicture}[thick,scale=\percentage]
    \node (l0) at (0,2) {$l_0$};
    \node (l1p) at (-1.5,0) {$\pi_1$};
    \node (lip) at (0,0) {$\dots$};
    \node (lnp) at (1.5,0) {$\pi_n$};
    \node (whole) at (0,-2) {$\mathbb R^3$};
   \draw  (l0) -- (l1p) -- (whole) -- (lip) -- (l0) -- (lnp) -- (whole);
\end{tikzpicture}
&
\begin{tikzpicture}[thick,scale=\percentage]
    \node (l) at (0,2) {$\nicefrac{l_0}G$};
    \node (lep) at (-1,0) {$\nicefrac{\pi_{even}}G$};
    \node (lop) at (1,0) {$\nicefrac{\pi_{odd}}G$};
    \node (whole) at (0,-2) {$\nicefrac{\mathbb R^3}G$};
    \draw (l) -- (lop) -- (whole) -- (lep) -- (l);
\end{tikzpicture}\\ \hline

\begin{tabular}{c}
$D_3(n)$\\
\mbox{for even } $n$ 
\end{tabular}
&
\begin{tikzpicture}[thick, scale=\percentage]
  \node (zero) at (0,2) {$O$};
  \node (l0) at (-1.5, 0) {$l_0$};
  \node (l1) at (-0.5, 0) {$l_1$};
  \node (li) at (0.5, 0) {$\dots$};
  \node (ln) at (1.5, 0) {$l_n$};
  \node (whole) at (0,-2) {$\mathbb R^3$};
  \draw (whole) -- (l0) -- (zero);
  \draw (whole) -- (l1) -- (zero);
  \draw (whole) -- (li) -- (zero);
  \draw (whole) -- (ln) -- (zero);
\end{tikzpicture}
&
\begin{tikzpicture}[thick, scale=\percentage]
  \node (zero) at (0,2) {$O$};
  \node (l0) at (-1.5, 0) {$\nicefrac{l_0}G$};
  \node (v) at (0, 0) {$\nicefrac{l_{even}}G$};
  \node (e) at (1.6, 0) {$\nicefrac{l_{odd}}G$};
  \node (whole) at (0,-2) {$\nicefrac{\mathbb R^3}G$};
  \draw (whole) -- (l0) -- (zero);
  \draw (whole) -- (e) -- (zero) -- (v) -- (whole);
\end{tikzpicture}\\ \hline

\begin{tabular}{c}
$D_3(n)$\\
\mbox{for odd } $n$ 
\end{tabular}
&
\begin{tikzpicture}[thick, scale=\percentage]
  \node (zero) at (0,2) {$O$};
  \node (l0) at (-1.5, 0) {$l_0$};
  \node (l1) at (-0.5, 0) {$l_1$};
  \node (li) at (0.5, 0) {$\dots$};
  \node (ln) at (1.5, 0) {$l_n$};
  \node (whole) at (0,-2) {$\mathbb R^3$};
  \draw (whole) -- (l0) -- (zero);
  \draw (whole) -- (l1) -- (zero);
  \draw (whole) -- (li) -- (zero);
  \draw (whole) -- (ln) -- (zero);
\end{tikzpicture}
&
\begin{tikzpicture}[thick, scale=\percentage]
  \node (zero) at (0,2) {$O$};
  \node (l0) at (-1, 0) {$\nicefrac{l_0}{G}$};
  \node (v) at (1, 0) {$\nicefrac{l_{>0}}{G}$};
  \node (whole) at (0,-2) {$\nicefrac{\mathbb R^3}G$};
  \draw (whole) -- (l0) -- (zero);
  \draw (whole) -- (v) -- (zero);
\end{tikzpicture}\\ \hline

\end{tabular} 

\caption{The subspace arrangements of $(D_3(n), \mu_3(n))$ and $D_3(n)$.}
\label{tab:fig:arr-D}
\label{D_nC_nx1}
\end{table}

\newpage

\begin{table}[h]
\begin{tabular}{ccc}
G &$\mathcal{A}^G$ & $\nicefrac{\mathcal A^G}G$ \\ \hline

$T_3$
&
\begin{tikzpicture}[thick, scale=\percentage]
  \node (zero) at (0,2) {$O$};
  \node (v1) at (-3, 0) {$v_1$};
  \node (v2) at (-2, 0) {$v_2$};
  \node (v3) at (-1, 0) {$v_3$};
  \node (v4) at (0, 0) {$v_4$};
  \node (e1) at (1, 0) {$e_1$};
  \node (e2) at (2, 0) {$e_2$};
  \node (e3) at (3, 0) {$e_3$};
  \node (whole) at (0,-2) {$\mathbb R^3$};
  \draw (whole) -- (v1) -- (zero);
  \draw (whole) -- (v2) -- (zero);
  \draw (whole) -- (v3) -- (zero);  
  \draw (whole) -- (v4) -- (zero);  
  \draw (whole) -- (e1) -- (zero);
  \draw (whole) -- (e2) -- (zero);
  \draw (whole) -- (e3) -- (zero);
\end{tikzpicture} &\hspace{30pt}
\begin{tikzpicture}[thick, scale=\percentage]
  \node (zero) at (0,2) {$O$};
  \node (v) at (-1, 0) {$\nicefrac{v}{T_3}$};
  \node (p) at (1, 0) {$\nicefrac{e}{T_3}$};
  \node (whole) at (0,-2) {$\nicefrac{\mathbb R^3}G$};
  \draw (whole) -- (p) -- (zero);
  \draw (whole) -- (v) -- (zero);
\end{tikzpicture}\\ \hline

$\J T_3$
&
\begin{tikzpicture}[thick, scale=\percentage]
  \node (zero) at (0,4) {$O$};
  \node (p12) at (1.6, 0) {$e_1^\perp$};
  \node (p13) at (2.3, 0) {$e_2^\perp$};
  \node (p14) at (3, 0) {$e_3^\perp$};
  \node (v1) at (-4, 2) {$v_1$};
  \node (v2) at (-3.3, 2) {$v_2$};
  \node (v3) at (-2.6, 2) {$v_3$};
  \node (v4) at (-1.9, 2) {$v_4$};
  \node (e1) at (-1.2, 2) {$e_1$};
  \node (e2) at (-0.5, 2) {$e_2$};
  \node (e3) at (0.2, 2) {$e_3$};
  \node (whole) at (0,-2) {$\mathbb R^3$};
  \draw (whole) -- (p12) -- (zero);
  \draw (whole) -- (p13) -- (zero);
  \draw (whole) -- (p14) -- (zero);
  \draw (whole) -- (v1) -- (zero);
  \draw (whole) -- (v2) -- (zero);
  \draw (whole) -- (v3) -- (zero);  
  \draw (whole) -- (v4) -- (zero);  
  \draw (whole) -- (e1) -- (zero);
  \draw (whole) -- (e2) -- (zero);
  \draw (whole) -- (e3) -- (zero);
\end{tikzpicture} &
\begin{tikzpicture}[thick, scale=\percentage]
  \node (zero) at (0,4) {$O$};
  \node (pi) at (1.4, 0) {$\nicefrac{e^\perp}{T_3}$};
  \node (v) at (-1.2, 2) {$\nicefrac{v}{T_3}$};
  \node (p) at (0, 2) {$\nicefrac{e}{T_3}$};
  \node (whole) at (0,-2) {$\nicefrac{\mathbb R^3}G$};
  \draw (whole) -- (pi) -- (zero);
  \draw (whole) -- (p) -- (zero);
  \draw (whole) -- (v) -- (zero);
\end{tikzpicture}\\\hline

$G=(Oct_3, T_3)=A_3$
&
\begin{tikzpicture}[thick,scale=\percentage]
  \node (zero) at (0,6) {$O$};
  \node (e1) at (-4,4) {$e_1$};
  \node (e2) at (-3,4) {$e_2$};
  \node (e3) at (-2,4) {$e_3$};
  \node (v1) at (-0.2,4) {$v_1$};
  \node (v2) at (1.2,4) {$v_2$};  
  \node (v3) at (2.6,4) {$v_3$};  
  \node (v4) at ( 4,4) {$v_4$};  
  \node (whole) at (0,0) {$\mathbb R^3$};

    \draw (e1) -- (zero) -- (e2);
    \draw (e3) -- (zero) -- (v1);
    \draw (v3) -- (zero) -- (v2);
    \draw (zero) -- (v4);
    
  \node (f1p) at ( 0.8,2) {$f_1^\perp$};
  \node (f2p) at ( 2.4,2) {$f_2^\perp$};
  \node (f3p) at ( 4  ,2) {$f_3^\perp$};
  \node (f4p) at (-4  ,2) {$f_4^\perp$};
  \node (f5p) at (-2.4,2) {$f_5^\perp$};
  \node (f6p) at (-0.8,2) {$f_6^\perp$};

    \draw (e1) -- (f1p);
    \draw (e1) -- (f4p);
    \draw (e2) -- (f2p);
    \draw (e2) -- (f5p);
    \draw (e3) -- (f3p);
    \draw (e3) -- (f6p);

    \draw (v1) -- (f4p);
    \draw (v1) -- (f5p);
    \draw (v1) -- (f6p);
    \draw (v2) -- (f2p); 
    \draw (v2) -- (f3p);
    \draw (v2) -- (f4p);
    \draw (v3) -- (f1p);
    \draw (v3) -- (f3p);
    \draw (v3) -- (f5p);
    \draw (v4) -- (f2p);
    \draw (v4) -- (f1p);
    \draw (v4) -- (f6p);

    \draw (whole) -- (f1p);
    \draw (whole) -- (f2p);
    \draw (whole) -- (f3p);
    \draw (whole) -- (f4p);
    \draw (whole) -- (f5p);
    \draw (whole) -- (f6p);
\end{tikzpicture}&
\begin{tikzpicture}[thick,scale=\percentage]
    
    \node (zero) at (0, 6) {$O$};
    \node (e) at (-1, 4) {$e$};
    \node (v) at ( 1, 4) {$v$};
    \node (p) at ( 0, 2) {$f^\perp$};
    \node (whole) at (0, 0) {$\nicefrac{\mathbb R^3}G$};

    \draw (whole) -- (p);
    \draw (zero) -- (e) -- (p) --(v) -- (zero);
\end{tikzpicture} \\ \hline

$G=Oct_3$
&
\begin{tikzpicture}[thick, scale=\percentage]
  \node (zero) at (0,2) {$O$};
  \node (v1) at (-4, 0) {$e_1$};
  \node (v) at (-3, 0) {$\dots$};
  \node (v6) at (-2, 0) {$e_3$};
  \node (e1) at (-1, 0) {$f_1$};
  \node (e) at (0, 0) {$\dots$};
  \node (e15) at (1, 0) {$f_{6}$};
  \node (f1) at (2, 0) {$v_1$};
  \node (f) at (3, 0) {$\dots$};
  \node (f10) at (4, 0) {$v_{4}$};
  \node (whole) at (0,-2) {$\mathbb R^3$};
  \draw (whole) -- (v1) -- (zero);
  \draw (whole) -- (v) -- (zero);
  \draw (whole) -- (v6) -- (zero);
  \draw (whole) -- (e1) -- (zero);
  \draw (whole) -- (e) -- (zero);
  \draw (whole) -- (e15) -- (zero);
  \draw (whole) -- (f1) -- (zero);
  \draw (whole) -- (f) -- (zero);
  \draw (whole) -- (f10) -- (zero);
\end{tikzpicture} &
\begin{tikzpicture}[thick, scale=\percentage]
  \node (zero) at (0,2) {$O$};
  \node (v) at (-2,0) {$\nicefrac{e}{G}$};
  \node (e) at ( 0,0) {$\nicefrac{f}{G}$};
  \node (f) at ( 2,0) {$\nicefrac{v}{G}$};  
  \node (whole) at (0,-2) {$\nicefrac{\mathbb R^3}G$};
  \draw (whole) -- (v) -- (zero);
  \draw (whole) -- (e) -- (zero);
  \draw (whole) -- (f) -- (zero);
\end{tikzpicture}\\ \hline
\end{tabular} 

\caption{The subspace arrangements of $T_3$, $\J T_3$, $A_3$, and $Oct_3$.}
\label{tab:fig:a3}
\label{tab:fig:arr-A_4}
\label{tab:fig:arr-A_4xI}
\label{table:oct-3}
\end{table}

\newpage

\begin{table}[h]
\begin{tabular}{ccc}
G &$\mathcal{A}^G$ & $\nicefrac{\mathcal A^G}G$ \\ \hline

\begin{tabular}{c}
$G=\J D_3(n)$\\
\mbox{for even } $n$ 
\end{tabular}
&
\begin{tikzpicture}[thick, scale=\percentage]
  \node (zero) at (0,4) {$O$};

  \node (l0) at (-4, 2)  {$l_0$};
  \node (l1) at (1, 2)   {$l_1$};
  \node (l) at  (1.9, 2) {$\dots$};
  \node (lk) at (2.8, 2) {$l_k$};
  \node (lk1) at(-2, 2)  {$l_{k+1}$};
  \node (ld) at (-1.1,2) {$\dots$};
  \node (l2k) at(-0.2,2) {$l_{2k}$};
  
  \node (l0p) at (-4,0) {$l_0^\perp$};
  \node (l1p) at (-2,0) {$l_1^\perp$};
  \node (lp) at (-1.1,0) {$\dots$};
  \node (lkp) at (-0.2,0) {$l_k^\perp$};
  \node (lk1p) at (1,0) {$l_{k+1}^\perp$};
  \node (ldp) at (1.9,0) {$\dots$};
  \node (l2kp) at (2.8,0) {$l_{2k}^\perp$};
  
  \node (whole) at (0,-2) {$\mathbb R^3$};
  
  \draw (zero) -- (l0);
  \draw (zero) -- (l1);
  \draw (zero) -- (l);
  \draw (zero) -- (lk);
  \draw (zero) -- (lk1);
  \draw (zero) -- (ld);
  \draw (zero) -- (l2k);

  \draw (l) -- (ldp);
  \draw (ld) -- (lp);
    
  \draw (l0) -- (l1p);
  \draw (l0) -- (lp);
  \draw (l0) -- (lkp);
  \draw (l0) -- (lk1p);
  \draw (l0) -- (ldp);
  \draw (l0) -- (l2kp);

  \draw (l0p) -- (l1);
  \draw (l0p) -- (l);
  \draw (l0p) -- (lk);
  \draw (l0p) -- (lk1);
  \draw (l0p) -- (ld);
  \draw (l0p) -- (l2k);

  \draw (l1) -- (lk1p);
  \draw (lk) -- (l2kp);
  \draw (lk1) -- (l1p);
  \draw (l2k) -- (lkp);

  \draw (whole) -- (l0p);
  \draw (whole) -- (l1p);
  \draw (whole) -- (lp);
  \draw (whole) -- (lkp);
  \draw (whole) -- (lk1p);
  \draw (whole) -- (ldp);
  \draw (whole) -- (l2kp);
\end{tikzpicture} 
&
\begin{tikzpicture}[thick, scale=\percentage]
  \node (zero) at (0,4) {$\nicefrac 0G$};

  \node (l0) at (-2, 2.2) {$\nicefrac{l_0}{G}$};
  \node (le) at (2, 2.2) {$\nicefrac{l_{even}}{G}$};
  \node (lo) at (0, 2.2) {$\nicefrac{l_{odd}}{G}$};

  \node (l0p) at (2,0) {$\nicefrac{l_0^\perp}{G}$};
  \node (lep) at (-2,0) {$\nicefrac{l_{even}^\perp}{G}$};
  \node (lop) at (0,0) {$\nicefrac{l_{odd}^\perp}{G}$};

  \node (whole) at (0,-1.8) {$\nicefrac{\mathbb R^3}{G}$};

  \draw (zero) -- (l0);
  \draw (zero) -- (le);
  \draw (zero) -- (lo);
 
  \draw (l0) -- (lep);
  \draw (l0) -- (lop);

  \draw (le) -- (lop);
  \draw (lo) -- (lep);

  \draw (le) -- (l0p) -- (lo);

  \draw (lop) -- (whole);
  \draw (lep) -- (whole);
  \draw (l0p) -- (whole);
\end{tikzpicture}\\ \hline

\begin{tabular}{c}
$G=\J D_3(n)$\\
\mbox{for odd } $n$ 
\end{tabular}
&
\begin{tikzpicture}[thick, scale=\percentage]
  \node (zero) at (0,4) {$O$};
  \node (pi1) at (-3, 0) {$l_1^\perp$};
  \node (pi) at (-2, 0) {$\dots$};
  \node (pik) at (-1, 0) {$l_{n}^\perp$};
  
  \node (l0) at (-2, 2) {$l_0$};
  \node (v1) at (1, 2) {$l_1$};
  \node (v) at (2, 2) {$\dots$};
  \node (vk) at (3, 2) {$l_{n}$};
  \node (whole) at (0,-2) {$\mathbb R^3$};
  
  \draw (whole) -- (v1);
  \draw (whole) -- (v);
  \draw (whole) -- (vk);
  
  \draw (l0) -- (pi1);
  \draw (l0) -- (pi);
  \draw (l0) -- (pik);
  
  \draw (zero) -- (l0);
  \draw (v1) -- (zero);
  \draw (v) -- (zero);
  \draw (vk) -- (zero);
  
  \draw (pi1) -- (whole);
  \draw (pi) -- (whole);
  \draw (pik) -- (whole);
\end{tikzpicture}
&
\begin{tikzpicture}[thick, scale=\percentage]
  \node (zero) at (0,4) {$\nicefrac 0G$};
  \node (pi) at (-1,0) {$\nicefrac{l_i^\perp}G$};
  
  \node (l0) at (-1, 2) {$\nicefrac{l_0}G$};
  \node (p) at (1, 2) {$\nicefrac{l_i}G$};
  \node (whole) at (0,-2) {$\nicefrac{\mathbb R^3}G$};
  
  \draw (whole) -- (pi);
  \draw (whole) -- (p);
 
  \draw (l0) -- (pi);
  
  \draw (p) -- (zero);
  \draw (l0) -- (zero);
\end{tikzpicture}\\ \hline

\begin{tabular}{c}
$(D_3({2n}),D_3(n))$\\
\mbox{for odd } $n$ 
\end{tabular}
&
\begin{tikzpicture}[thick,scale=\percentage]
    \node (zero) at (0,4) {$O$};
    \node (l0) at (-4,2) {$l_0$};
    \node (l1) at (-1,2) {$l_1$};
    \node (ld1) at (0,2) {$\dots$};
    \node (ln) at  (1,2) {$l_n$};
    \node (ln2) at (2,2) {$l_{n+2}$};
    \node (ld2) at (3,2) {$\dots$};
    \node (llo) at (4,2) {$l_{2n-1}$};
    \node (l2p) at (-1,0) {$l_{2}^\perp$};
    \node (lpd1) at (0,0) {$\dots$};
    \node (ln-1p) at (1,0) {$l_{n-1}^\perp$};
    \node (ln+1p) at (-4,0) {$l_{n+1}^\perp$};
    \node (lpd2) at (-3,0) {$\dots$};
    \node (l2np) at (-2,0) {$l_{2n}^\perp$};
    \node (l0p) at (4,0) {$l_0^\perp$};
    \node (whole) at (0,-2) {$\mathbb R^3$};
    \draw (whole) -- (lpd2) -- (ld1) -- (zero) -- (ld2) -- (lpd1) -- (whole) -- (l2np) -- (ln) -- (l0p) -- (ln) -- (zero) -- (l0) -- (l2p) -- (whole) -- (l0p) -- (llo) -- (zero) -- (l1) -- (l0p) -- (ln2) -- (zero) ;
    \draw (ln-1p) -- (whole) -- (ln+1p) -- (l1) ;
    \draw (ln) -- (l2np) -- (l0) -- (ln-1p) -- (llo) ;
    \draw (ln+1p) -- (l0) -- (lpd1);
    \draw (ln2) -- (l2p);
    \draw (l0) -- (lpd2);
    \draw (ld1) -- (l0p) -- (ld2);
\end{tikzpicture}&
\begin{tikzpicture}[thick,scale=\percentage]
    \node (zero) at (0,4) {$O$};
    \node (l0) at (-1,2) {$\nicefrac{l_0}G$};
    \node (lo) at  (1,2) {$\nicefrac{l_{odd}}G$};
    \node (lep) at (-1,0) {$\nicefrac{l_{even}^\perp}G$};
    \node (l0p) at (1,0) {$\nicefrac{l_0^\perp}G$};
    \node (whole) at (0,-2) {$\nicefrac{\mathbb R^3}G$};
    \draw (l0) -- (lep) -- (whole) -- (l0p) -- (lo) -- (zero) -- (l0);
    \draw (lep) -- (lo);
\end{tikzpicture}\\ \hline

\end{tabular} 

\caption{The subspace arrangements of $\J D_3(n)$, and $(D_3(2n), D_3(n))$ for $n$ odd.}
\label{tab:fig:arr-D_evenxI}
\label{tab:fig:arr-D_oddxI}
\label{table:D2nDn-n-odd}
\end{table}
\newpage
\def\percentage{0.6}
\begin{table}[h]
\begin{tabular}{ccc}
G &$\mathcal{A}^G$ & $\nicefrac{\mathcal A^G}G$ \\ \hline

\begin{tabular}{c}
$(D_3({2n}),D_3(n))$\\
\mbox{for even } $n$ 
\end{tabular}
&
\begin{tikzpicture}[thick,scale=\percentage]
    \node (zero) at (0,4) {$O$};
    \node (l0) at (-1.5,2) {$l_0$};
    \node (l1) at (0.5,2) {$l_1$};
    \node (lo) at (1.5,2) {$\dots$};
    \node (llo) at (2.5,2) {$l_{2n-1}$};
    \node (l2p) at (-2.5,0) {$l_{2}^\perp$};
    \node (lep) at (-1.5,0) {$\dots$};
    \node (llep) at (-0.5,0) {$l_{2n}^\perp$};
    \node (whole) at (0,-2) {$\mathbb R^3$};
    \draw (whole) -- (llep) -- (l0) -- (lep) -- (whole) -- (lo) -- (zero) -- (l0) -- (l2p) -- (whole) -- (llo) -- (zero) -- (l1) -- (whole);
\end{tikzpicture}&
\begin{tikzpicture}[thick,scale=\percentage]
    \node (zero) at (0,4) {$O$};
    \node (l0) at (-1,2) {$\nicefrac{l_0}G$};
    \node (lo) at (1,2) {$\nicefrac{l_{odd}}G$};
    \node (lep) at (-1,0) {$\nicefrac{l_{even}^\perp}G$};
    \node (whole) at (0,-2) {$\nicefrac{\mathbb R^3}G$};
    \draw (l0) -- (lep) -- (whole) -- (lo) -- (zero) -- (l0);
\end{tikzpicture}\\ \hline

\def\percentage{0.8}

$\J Oct_3$
&
\begin{tikzpicture}[thick, scale=0.6]
  \node (zero) at (0,5) {$O$};
  \node (e1) at (-6, 3) {$e_1$};
  \node (e2) at (-5, 3) {$e_2$};
  \node (e3) at (-4, 3) {$e_3$};
  \node (f1) at (-3, 3) {$f_1$};
  \node (f2) at (-2, 3) {$f_2$};
  \node (f3) at (-1, 3) {$f_3$};
  \node (f4) at ( 0, 3) {$f_4$};
  \node (f5) at ( 1, 3) {$f_5$};
  \node (f6) at ( 2, 3) {$f_6$};
  \node (v1) at ( 3, 3) {$v_1$};
  \node (v2) at ( 4, 3) {$v_2$};
  \node (v3) at ( 5, 3) {$v_3$};
  \node (v4) at ( 6, 3) {$v_4$};
  
  \node (e1p) at (-6  ,-2) {$e_1^\perp$};
  \node (e2p) at (-4.5,-2) {$e_2^\perp$};
  \node (e3p) at (-3  ,-2) {$e_3^\perp$};
  \node (f1p) at ( 3  ,-2) {$f_1^\perp$};
  \node (f2p) at ( 4.5,-2) {$f_2^\perp$};
  \node (f3p) at ( 6  ,-2) {$f_3^\perp$};
  \node (f4p) at (-1.5,-2) {$f_4^\perp$};
  \node (f5p) at ( 0  ,-2) {$f_5^\perp$};
  \node (f6p) at (1.5  ,-2) {$f_6^\perp$};
  \node (whole) at (0,-4) {$\mathbb R^3$};

    \draw (e1) -- (e2p);
    \draw (e1) -- (e3p);
    \draw (e2) -- (e3p);
    \draw (e2) -- (e1p);
    \draw (e3) -- (e2p);
    \draw (e3) -- (e1p);
    \draw (e1) -- (f1p);
    \draw (e1) -- (f4p);
    \draw (e2) -- (f2p);
    \draw (e2) -- (f5p);
    \draw (e3) -- (f3p);
    \draw (e3) -- (f6p);

    \draw (v1) -- (f4p);
    \draw (v1) -- (f5p);
    \draw (v1) -- (f6p);
    \draw (v2) -- (f2p); 
    \draw (v2) -- (f3p);
    \draw (v2) -- (f4p);
    \draw (v3) -- (f1p);
    \draw (v3) -- (f3p);
    \draw (v3) -- (f5p);
    \draw (v4) -- (f2p);
    \draw (v4) -- (f1p);
    \draw (v4) -- (f6p);

    \draw (f1) -- (f4p);
    \draw (f2) -- (f5p);
    \draw (f3) -- (f6p);
    \draw (f4) -- (f1p);
    \draw (f5) -- (f2p);
    \draw (f6) -- (f3p);

    \draw (f1) -- (e1p);
    \draw (f4) -- (e1p);
    \draw (f2) -- (e2p);
    \draw (f5) -- (e2p);
    \draw (f3) -- (e3p);
    \draw (f6) -- (e3p);

    \draw (e1) -- (zero);
    \draw (e2) -- (zero);
    \draw (e3) -- (zero);
    \draw (v1) -- (zero);
    \draw (v2) -- (zero);
    \draw (v3) -- (zero);
    \draw (v4) -- (zero);
    \draw (f1) -- (zero);
    \draw (f2) -- (zero);
    \draw (f3) -- (zero);
    \draw (f4) -- (zero);
    \draw (f5) -- (zero);
    \draw (f6) -- (zero);
  
    \draw (whole) -- (e1p);
    \draw (whole) -- (e2p);
    \draw (whole) -- (e3p);
    \draw (whole) -- (f1p);
    \draw (whole) -- (f2p);
    \draw (whole) -- (f3p);
    \draw (whole) -- (f4p);
    \draw (whole) -- (f5p);
    \draw (whole) -- (f6p);
\end{tikzpicture}&
\begin{tikzpicture}[thick, scale=0.8]
  \node (zero) at (0,4.5) {$O$};
  \node (e) at (-2,2.5) {$\nicefrac eG$};
  \node (f) at ( 0,2.5) {$\nicefrac fG$};
  \node (v) at ( 2,2.5) {$\nicefrac vG$};  
  \node (ep) at (-1,0) {$\nicefrac{e^\perp}G$};  
  \node (fp) at ( 1,0) {$\nicefrac{f^\perp}G$};  
   
  \node (whole) at (0,-2) {$\nicefrac{\mathbb R^3}G$};
  \draw (f) -- (zero) -- (e) -- (ep)  -- (whole) -- (fp) -- (e);
    \draw (zero) -- (v) -- (fp) -- (f) -- (ep);
\end{tikzpicture}\\ \hline

$Ico_3$
&
\begin{tikzpicture}[thick, scale=0.7]
  \node (zero) at (0,2) {$O$};
  \node (v1) at (-4, 0) {$v_1$};
  \node (v) at (-3, 0) {$\dots$};
  \node (v6) at (-2, 0) {$v_6$};
  \node (e1) at (-1, 0) {$e_1$};
  \node (e) at (0, 0) {$\dots$};
  \node (e15) at (1, 0) {$e_{15}$};
  \node (f1) at (2, 0) {$f_1$};
  \node (f) at (3, 0) {$\dots$};
  \node (f10) at (4, 0) {$f_{10}$};
  \node (whole) at (0,-2) {$\mathbb R^3$};
  \draw (whole) -- (v1) -- (zero);
  \draw (whole) -- (v) -- (zero);
  \draw (whole) -- (v6) -- (zero);
  \draw (whole) -- (e1) -- (zero);
  \draw (whole) -- (e) -- (zero);
  \draw (whole) -- (e15) -- (zero);
  \draw (whole) -- (f1) -- (zero);
  \draw (whole) -- (f) -- (zero);
  \draw (whole) -- (f10) -- (zero);
\end{tikzpicture} &
\begin{tikzpicture}[thick, scale=0.7]
  \node (zero) at (0,2) {$O$};
  \node (v) at (-2, 0) {$\nicefrac{v}{G}$};
  \node (e) at (0, 0) {$\nicefrac{e}{G}$};
  \node (f) at (2, 0) {$\nicefrac{f}{G}$};  
  \node (whole) at (0,-2) {$\nicefrac{\mathbb R^3}G$};
  \draw (whole) -- (v) -- (zero);
  \draw (whole) -- (e) -- (zero);
  \draw (whole) -- (f) -- (zero);
\end{tikzpicture}\\ \hline

\def\percentage{0.8}

$H_3=\J Ico_3$
&
\begin{tikzpicture}[thick,scale=0.8]
  \node (zero) at (0,4) {$O$};
   
  \node (v1) at (-3.7, 2) {$v_1$};
  \node (v) at (-3, 2) {$\dots$};
  \node (vv) at (-2.3, 2) {$\dots$};
  \node (v6) at (-1.6, 2) {$v_6$};
  \node (e1) at (-0.9, 2) {$e_1$};
  \node (e) at (-0.2, 2) {$\dots$};
  \node (ee) at (0.5, 2) {$\dots$};
  \node (e15) at (1.2, 2) {$e_{15}$};
  \node (f1) at (1.9, 2) {$f_1$};
  \node (f) at (2.6, 2) {$\dots$};
  \node (ff) at (3.3, 2) {$\dots$};
  \node (f10) at (4, 2) {$f_{10}$};

  \node (p1) at (-3, 0) {$\pi_1$};
  \node (p) at (-1, 0) {$\dots$};
  \node (pp) at (1, 0) {$\dots$};
  \node (p15) at (3, 0) {$\pi_{15}$};

  \node (whole) at (0,-1.5) {$\mathbb R^3$};
  
  \draw (zero) -- (v1);
  \draw (zero) -- (v);
  \draw (zero) -- (vv);
  \draw (zero) -- (v6);
  \draw (zero) -- (e1);
  \draw (zero) -- (e);
  \draw (zero) -- (ee);
  \draw (zero) -- (e15);
  \draw (zero) -- (f1);
  \draw (zero) -- (f);
  \draw (zero) -- (ff);
  \draw (zero) -- (f10);
  
  \draw (p1) -- (v6);
  \draw (p1) -- (v);
  \draw (p1) -- (e);
  \draw (p1) -- (ee);
  \draw (p1) -- (f1);
  \draw (p1) -- (f);
  
  \draw (p15) -- (v1);
  \draw (p15) -- (vv);
  \draw (p15) -- (e);
  \draw (p15) -- (ee);
  \draw (p15) -- (f10);
  \draw (p15) -- (ff);
  
  \draw (p) -- (v);
  \draw (p) -- (vv);
  \draw (p) -- (e1);
  \draw (p) -- (e15);
  \draw (p) -- (f);
  \draw (p) -- (ff);
  
  \draw (pp) -- (v);
  \draw (pp) -- (vv);
  \draw (pp) -- (e1);
  \draw (pp) -- (e15);
  \draw (pp) -- (f);
  \draw (pp) -- (ff);
  
  \draw (p1) -- (whole);
  \draw (p) -- (whole);
  \draw (pp) -- (whole);
  \draw (p15) -- (whole);
\end{tikzpicture} 
&
\begin{tikzpicture}[thick, scale=0.8]
  \node (zero) at (0,4) {$O$};
  \node (pi) at (0, 0) {$\nicefrac{\pi}{...}$};
  \node (v) at (-2, 2) {$\nicefrac{v}{...}$};
  \node (e) at (0, 2) {$\nicefrac{e}{...}$};
  \node (f) at (2, 2) {$\nicefrac{f}{...}$};  
  \node (whole) at (0,-1.6) {$\nicefrac{\mathbb R^3}{G}$};
  \draw (whole) -- (pi) -- (v) -- (zero);
  \draw (pi) -- (e) -- (zero);
  \draw (pi) -- (f) -- (zero);
\end{tikzpicture}
\end{tabular} 

\caption{The subspace arrangements of $\J Oct_3$, $Ico_3$, $H_3$, and $(D_3(2n), D_3(n))$ for $n$ even.}
\label{table:D2nDn-n-even}
\label{tab:fig:arr-S_4xI}
\label{tab:fig:arr-A_5}
\label{tab:fig:arr-A_5xI}
\end{table}

\newpage
\section{Complete Classification in Rank Three}\label{sec:classification}
In this section, we compute the subspace arrangements $\mathcal A^G$ for all groups of rank three.
We show the following result about groups strictly generated in codimension two.
\begin{TheoM}\label{thm-L_n}
Suppose $G$ is strictly generated in codimension two. 
Then $\mathcal{L}(\mathcal{A}^G)$ is isomorphic, as a poset, to $\mathcal L_n$ for $n\neq 2$.
More precisely, for any positive integer $n\neq 2$, there exists a finite linear group $G$ in ${GL}_3(\mathbb{R})$ strictly generated in codimension two such that $\mathcal{L}(\mathcal{A}^G)\cong \mathcal L_n$. ($\mathcal L_n$ is shown in \Cref{fig:Ln}.)
\end{TheoM}
\begin{proof}
By \Cref{pro:no-cod-one}, if $G$ is strictly generated in codimension two, then there are no planes in the arrangement $\mathcal{A}^G$.
Moreover, there is at least a line in $\mathcal{A}^G$ otherwise $G$ is not generated in codimension two.
Hence, $\mathcal{A}^G$ is poset-wise isomorphic to $\mathcal L_n$.
   
If $n=1$, then set $G=\J \mu_2(2k+1)$, see \Cref{sec:mu-2xI-in-GL-3}. 
If $n>2$, we set $G=D_3(n-1)$, see \Cref{sec:D3n-in-GL-3}.
    
Finally, we observe that $\mathcal{A}^G$ cannot have only two lines. 
Indeed, if there are two rotations $r_1, r_2\in G$ the product $r_1 r_2$ is a rotation distinct from $r_1$ and $r_2$.
\end{proof}

Moreover, we also classify the intersection lattice of the quotient arrangement.
\begin{TheoM}
Let $G$ be a rank three linear group strictly generated in codimension two.
The intersection lattice of the closures of the Luna strata of $\nicefrac{V}{G}$ is isomorphic to $\mathcal L_n$ for $n\in\{2,3\}$.    
\end{TheoM}
\begin{proof}
If $G$ is a rank three linear group strictly generated in codimension two, then it is one of the group in the list of \Cref{pro:Klein-class}. 
Hence, $G$ is a subgroup of the group of symmetries of a regular polygon or a regular polyhedron.
Thus we have two facts:\begin{itemize}
    \item The group $G$ acts transitively on the vertices, edges, and in the case of the polyhedron, facets;
    \item Any axis of a rotation in $\mathcal{A}^G$ is an axis of symmetry, therefore associate to the regular polygon (or the regular polyhedron).
\end{itemize}
Thus, we can only have three possible orbits in $\nicefrac{\mathcal{A}^G}{G}$ and the latter is isomorphic to $\mathcal{L}_n$ for $n\leq 3$.
It is a simple check of \Cref{sec:classification} that $\mathcal{L}_2$ and $\mathcal{L}_3$ occur: see \Cref{tab:fig:arr-A_4}.
\end{proof}

The rest of the section is devoted to the study of the arrangement of subspaces of a finite linear groups in $GL_3(\mathbb{R})$. 
Moreover, we describe every inclusion in \Cref{fig:all-arrangements} case by case, compute the normal subgroups $\mathcal{R}_1 G$, $\mathcal{R}_2 G$, 
and fill \Cref{allGroupData}.
As a byproduct, we prove that $A_3=(Oct_3,T_3)$, $BC_3=\J Oct_3$, $H_3=\J Ico_3$, $I_2(n)\times\triv=(D_3(n),\mu_3(n))$, and
\begin{align*}
    I_2(n)\times A_1&=\begin{cases}\J D_3(n) &\text{if $n$ is even,}\\
                        (D_3(2n),D_3(n))    &\text{if $n$ is odd.}
                    \end{cases}
\end{align*}

We start with two useful lemmas.

\begin{lem}\label{lem:FromRotationToReflection}
Consider a non-trivial finite order element $g\in SO_3(\mathbb R)$.
Then $g$ is a rotation by some angle $\theta$ around the axis $V^{\langle g\rangle}$.
Then \begin{align*}
    V^{\langle \J g\rangle} =\begin{cases} (V^{\langle g\rangle})^\perp &\text{if }\theta=\pi,\\
                                0           &\text{otherwise.}
                            \end{cases}
\end{align*}
In particular, $\J g$ is a reflection if and only if $\theta=\pi$.
\end{lem}
\begin{proof}
Let $p$ be a non-zero vector in $V^{\langle g\rangle}$.
Let $\{e,f\}$ be an orthonormal basis of $p^\perp$.
Then we have
\begin{align*}
        &\quad\,g=\begin{bmatrix} 1 & 0         & 0\\
                    0               &\cos\theta &-\sin \theta \\
                    0               &\sin\theta &\cos\theta
        \end{bmatrix}\\
\implies&\J g=\begin{bmatrix} -1  & 0                 & 0\\
                0               &\cos(\theta+\pi)   &-\sin(\theta+\pi)\\
                0               &\sin(\theta+\pi) &\cos(\theta+\pi) 					
            \end{bmatrix}
\end{align*}
in the basis $\{p,e,f\}$.
It follows that $\J g$ fixes a nontrivial subspace if and only if $\theta=\pi$. 
Further, if $\theta=\pi$, then $V^{\langle\J g\rangle}$ is the plane spanned by $\{e,f\}$, i.e., the plane $(V^{\langle g\rangle})^\perp$.
\end{proof}

\begin{pro}\label{pro:jh-codim2}
Let $H$ be a linear subgroup of $SO_3(\mathbb R)$.
The linear group $\J H$ is generated in codimension two if and only if $H$ contains a rotation of angle $\pi$.
\end{pro}
\begin{proof}
We know from \Cref{flgso3} that $H$ is strictly generated in codimension two, hence $H\subset\mathcal R_2\J H$.
Further, since $H$ has index two in $\J H$, and $\J\not\in\mathcal R_2\J H$, we deduce\[
    H\subset\mathcal R_2\J H\subsetneq\J H\implies\mathcal R_2\J H=H.
\]
Let $K$ be the subgroup of $\J H$ generated by the rotations and reflections in $\J H$.
Then $K=\J H$ if and only if $\J\in K$.

Suppose $H$ contains a rotation $r$ of angle $\pi$.
Then $\J r$ is a reflection, hence $\J r\in K$.
Further, since $r\in H\subset K$, we have $\J\in K$.

Conversely, suppose $\J\in K$.
Then $\J=sr$ where $r$ is a rotation, and $s$ is a product of reflections in $H$.
Suppose $s=s_1\ldots s_k$, and set $r_1=\J s_1$.
Then $\det r_1=\det(\J s_1)=1$; hence $r_1$ is a rotation in $H$, and $\J r_1=\J^2s_1=s_1$ is a reflection.
It follows from \Cref{lem:FromRotationToReflection} that $r_1$ is a rotation of angle $\pi$.
\end{proof}




\subsection{The group $\mu_3(n)$}\label{sec:mu-2-in-GL-3}
The group $G=\mu_3(n)$ is not essential. 
The subspace arrangement of $G$, in \Cref{tab:fig:arr-C_n}, is deduced from the arrangement of the essential rank $2$ group $\mu_2(n)$.
We note that $\subweyl G$ is trivial, $\mathcal R_2G=G$, and also $\mu_2(n)=\mathcal R_2 I_2(n)$, and $\mu_3(n) =\mathcal R_2(I_2(n)\times\triv)$.


\subsection{The group $\J\mu_3(n)$}\label{sec:mu-2xI-in-GL-3}
The generator $\J$ makes the representation essential and fixes the origin.
Let $l_0$ be the one-dimensional subspace on which the group acts trivially.

Suppose first that $n$ is even, so that the group $G=\J\mu_3(n)$ contains a rotation of angle $\pi$; by \Cref{lem:FromRotationToReflection}, there is a reflection in $\J\mu_3(n)$ across the plane $l_0^\perp$.
We deduce from \Cref{lem:FromRotationToReflection} that $\J\mu_3(n)=\mu_2(n)\times A_1$, $\mathcal R_1(\mu_2(n)\times A_1)=\triv[2]\times A_1$, and $\mathcal R_2(\mu_2(n)\times A_1)=\mu_3(n)$.

We now consider the case where $n$ is odd, so that there are no rotations with angle $\pi$ in $\mu_2(n)$. 
It follows from \Cref{pro:jh-codim2} that $\J\mu_3(n)$ is \emph{not} generated in codimension $2$, and that the subspace arrangement of $\J\mu_3(n)$ is made by $O<l_0< \mathbb{R}^3$, see \Cref{tab:fig:arr-C_oddxI}.
Further, $\subweyl G=\triv[3]$, and $\mathcal R_2G=\mu_3(n)$.




\subsection{The mixed group $(\mu_3(2n),\mu_3(n))$}
Suppose first that $n$ is even.
The group $G=(\mu_3(2n),\mu_3(n))$ has the same arrangements as $\J\mu_3(n)$, see \Cref{c2ncne}. 
By \Cref{def:mixedgroup}, we have
$$(\mu_3(2n),\mu_3(n)) = \mu_3(n)\sqcup\left\{\J g\mmid g\in \mu_3(2n)\backslash \mu_3(n)\right\}.$$ 
The rotation by the angle $\pi$ is not among the elements $g$ that we use for $\J g$.
It is clear from the arrangement that $G$ is \emph{not} generated in codimension two.
Indeed $\mathcal R_1G=\triv[3]$, and $\mathcal R_2G=\mu_3(n)$.

Suppose now that $n$ is odd, so that $(\mu_3(2n),\mu_3(n))$ has a reflection across the plane $l_0^\perp$. 
This group is generated in codimension two, with $\subweyl G=\triv[2]\times A_1$, and $\mathcal R_2G=\mu_3(n) $.
The arrangement $\mathcal A^G$ is given in \Cref{c2ncno}.
It is similar to the arrangement of $\J\mu_3(k)$ for even $k$. 


\subsection{The group $D_3(n)$}\label{sec:D3n-in-GL-3}
The linear group $G=D_3(n)$ is strictly generated  in codimension $2$, so that $\subweyl G=\triv[3]$, and $\mathcal R_2G=G$.

Let $l_1, l_2,\cdots,l_n$ be the lines passing through a pair of opposite vertices of a regular $2n$-gon $P$; and $l_0$ the line perpendicular to the plane containing $P$. 
There are $n$ rotations $r_1, \dots, r_n$ of angle $\pi$ around the line $l_1, \dots, l_n$. 
Further there is a rotation $r_0$ of angle $\nicefrac{\pi}{n}$ around the polar line $l_0$.
\begin{itemize}
    \item If $n$ is even, the arrangement splits into $3$ orbits: a single orbit for the \emph{polar} line $l_0$, the orbit $l_{even}=\left\{l_i\mmid i>0, i\text{ even}\right\}$, and the orbit $l_{odd}=\left\{l_i\mmid i\text{ odd}\right\}$.
    \item If $n$ is odd, the arrangement splits into $2$ orbits: a single orbit for the \emph{polar} line $l_0$, and the orbit $l=\left\{l_i\mmid i\neq 0\right\}$, and the orbit $l_{odd}=\left\{l_i\mmid i\text{ odd}\right\}$.
\end{itemize}
The arrangements are shown in \Cref{tab:fig:arr-D}.



\subsection{The group $G=\J D_3(n)$} 
When we add $\J$ to the Dihedral group, we need to distinguish again two cases.

Suppose first that $n=2k$ for some $k$. 
We want to show that $\J D_3(2k)$ is a reflection group, and precisely this is $I_2(2k)\times A_1$.

Let $l_1, l_2,\cdots,l_{2k}$ be the lines passing through a pair of opposite vertices of the $4k$-gon.
There are $2k$ rotations $r_1, \dots, r_{2k}$  by $\pi$ around the line $l_1, \dots, l_{2k}$. 
Further, there is a rotation $r_0$ by $\pi$ around the polar line $l_0$, because $n$ is even.
It follows that $\J r_i$, with $0\leq i \leq 2k$, is a reflection with respect the plane $l_i^\perp$, which contains exactly two lines from the arrangement: $l_0$ and $l_{i\pm k}$.
Thus, the group $I_2(2k)\times A_1$ is a subgroup of $\J D_3(2k)$, with the same order.

The arrangement $\mathcal A^G$ contains the plane $l_0^\perp$ and the lines $l_1,\cdots,l_{2k}$, see \Cref{tab:fig:arr-D_evenxI}.
The rotation around the line $l_0$ acts transitively on the reflection planes $l_{2k}^\perp$ with $2k\neq 0$ and similarly on $l_{odd}^\perp$. 
The group $G$ is generated in codimension $1$. Further, we have\begin{align*}
    \subweyl G      &=I_2(2k)\times A_1,\\
    \mathcal R_2(G) &=D_3(2k).
\end{align*}
Hence, we can read the one dimensional orbits directly from the arrangements of $D_3(2k)$.
The orbit arrangements of $D_3(2k)$ for $k$ odd is also shown in \Cref{tab:fig:arr-D}.
The quotient arrangement depends on the parity of $k$ as it is clear from \Cref{tab:fig:arr-D}, see \Cref{tab:fig:arr-D_evenxIorbits}. 


  
  
  

    




\begin{table}[h]
\begin{tabular}{cc}
$\nicefrac{\mathcal{A}^G}G$ &\qquad\qquad $\nicefrac{\mathcal A^G}G$\\

\begin{tikzpicture}[thick, scale=\percentage]
  \node (zero) at (0,4) {$\nicefrac 0G$};

  \node (l0) at (-2, 2.2) {$\nicefrac{l_0}{G}$};
  \node (le) at (2, 2.2) {$\nicefrac{l_{even}}{G}$};
  \node (lo) at (0, 2.2) {$\nicefrac{l_{odd}}{G}$};

  \node (l0p) at (2,0) {$\nicefrac{l_0^\perp}{G}$};
  \node (lep) at (-2,0) {$\nicefrac{l_{even}^\perp}{G}$};
  \node (lop) at (0,0) {$\nicefrac{l_{odd}^\perp}{G}$};

  \node (whole) at (0,-1.8) {$\nicefrac{\mathbb R^3}{G}$};

  \draw (zero) -- (l0);
  \draw (zero) -- (le);
  \draw (zero) -- (lo);
 
  \draw (l0) -- (lep);
  \draw (l0) -- (lop);

  \draw (le) -- (lop);
  \draw (lo) -- (lep);

  \draw (le) -- (l0p) -- (lo);

  \draw (lop) -- (whole);
  \draw (lep) -- (whole);
  \draw (l0p) -- (whole);
\end{tikzpicture} &\qquad
\begin{tikzpicture}[thick, scale=\percentage]
  \node (zero) at (0,4) {$\nicefrac 0G$};

  \node (l0) at (-2, 2.2) {$\nicefrac{l_0}{G}$};
  \node (le) at (2, 2.2) {$\nicefrac{l_{even}}{G}$};
  \node (lo) at (0, 2.2) {$\nicefrac{l_{odd}}{G}$};

  \node (l0p) at (2,0) {$\nicefrac{l_0^\perp}{G}$};
  \node (lep) at (0,0) {$\nicefrac{l_{even}^\perp}{G}$};
  \node (lop) at (-2,0) {$\nicefrac{l_{odd}^\perp}{G}$};

  \node (whole) at (0,-1.8) {$\nicefrac{\mathbb R^3}{G}$};

  \draw (zero) -- (l0);
  \draw (zero) -- (le);
  \draw (zero) -- (lo);
 
  \draw (l0) -- (lop);
  \draw (l0) -- (lep);
  \draw (le) -- (lep);
  \draw (lo) -- (lop);

  \draw (le) -- (l0p) -- (lo);

  \draw (lop) -- (whole);
  \draw (lep) -- (whole);
  \draw (l0p) -- (whole);
\end{tikzpicture}
\end{tabular}
\caption{The arrangement quotients of $G=\J D_{2k}$ for $k$ odd and $k$ even respectively.}
\label{tab:fig:arr-D_evenxIorbits}
\end{table}

Let now consider the linear group $G=\J D_3(n)$, where $n=2k+1$ for some $k$.
Let $l_1,\cdots,l_n$ be the lines determined by the vertices of a regular $n$-gon and the origin; let $l_0$ denote the polar line.
For $1\leq i\leq n$, we denote by $r_i$ the rotation of angle $\pi$ around the axis $l_i$.
Then $\J r_i$ is a reflection across the plane $l_i^\perp$.
The polar line $l_0$ is contained in $l_i^\perp$. We note that $l_i^\perp$ does not contain $l_j$ for any $j$, because $n$ is odd.
The subspace arrangement is shown in \Cref{tab:fig:arr-D_oddxI}.
This group is generated in codimension $2$, with \begin{align*}
    \subweyl G      &=I_2(n)\times\triv.\\
    \mathcal R_2G   &=D_3(n).
\end{align*}

  
  
  
  
  
  
  
 
  

\subsection{The mixed group $(D_3(n),\mu_3(n) )$}\label{sec:mixed-d3n-mu3n}
By \Cref{def:mixedgroup}, 
\[
(D_3(n),\mu_3(n) ) = \mu_3(n) \sqcup\left\{\J g\mmid g\in D_3(n)\backslash \mu_3(n) \right\}.
\]
Observe that this group equals $I_2(n)\times\triv$. 
It follows that $\subweyl G=G$ and $\mathcal R_2G=\mu_3(n)$.

Let $\pi_0$ be the plane on which $I_2(n)$ acts, and let $P$ be the regular $2n$-polygon in $\pi_0$. 
For $1\leq i\leq n$, let $\pi_i$ be the plane containing $l_0=\pi_0^\perp$, and the $i^{th}$ vertex of $P$.
The subspace arrangement $\mathcal A^G$ contains the line $l_0$ and the $n$ planes $\pi_1,\ldots,\pi_n$, see \Cref{D_nC_nx1}.



\subsection{The mixed group $(D_3({2n}),D_3(n))$} 
Suppose first that $n$ is odd. 
We show that in this case $(D_3({2n}),D_3(n))=I_2(n)\times A_1$.
Since $n$ is odd, the rotation $r_0$ with angle $\pi$ and axis $l_0$ belongs to $D_3(2n)\backslash D_3(n)$.
Hence the reflection $\J r_0$ belongs to $G$, i.e., $\triv[2]\times A_1\subset (D_3({2n}),D_3(n))$.
Let $l_1,\cdots, l_{2n}$ be the lines in the arrangement of $\J D_3(2n)$, and $r_i$ the rotation of angle $\pi$ with axis $l_i$.
For $i$ even, $\J r_i$ is a reflection across $l_i^\perp=\langle l_0,l_{i+n}\rangle$.
Hence the subgroup generated by the $l_i^\perp$ for $i$ even is precisely $I_2(n)\times\triv$.
It follows by comparing order that $(D_3({2n}),D_3(n))=I_2(n)\times A_1$.
The lattice of intersection of the related subspace arrangement is computed in \Cref{table:D2nDn-n-even}.
We leave to the reader to check that the quotient arrangement is the one described in \Cref{table:D2nDn-n-even}; we just remark that $l_i^\perp=\langle l_0,l_{i+n}\rangle$.

When $n$ is even, $r_0$ does not belong in $D_3({2n})\backslash D_3(n)$.
Moreover, the reflecting hyperplane of any reflection in $\left\{\J g\mmid g\in D_3({2n})\backslash D_3(n)\right\}$ contains only the line $l_0$ in the arrangement of $D_3({2n})\backslash D_3(n)$.

Hence, the group $(D_3({2n}),D_3(n))$ is generated in codimension $2$, and we have $\subweyl G=I_2(n)\times\triv$, and $\mathcal R_2G=D_3(n)$.

This argument, along with the observation that the lines $l_{2j}$ in the arrangement of $D_3(n)$ are missing, explain the poset of the quotient arrangement. 



\subsection{The group $T_3$}\label{sec:T_3}
The group $T_3$ is the group generated by the rotational symmetries of the regular tetrahedron $P$ in $\mathbb R^3$.
Let $v_1, v_2, v_3, v_4$ be the lines containing the origin and a vertex of $P$; and let $e_1, e_2, e_3$ be the lines passing through the midpoints of two opposite edges of $P$. 
%
%
The arrangements of the linear group $T_3$ are shown in \Cref{tab:fig:arr-A_4}.
Since $G$ is generated by rotations, we have $\subweyl G=\triv[3]$, and $\mathcal R_2G=G$.


\subsection{The group $\J T_3$}
The rotations $r_i$ in $T_3$ with axes $e_i$ has angle $\pi$, hence $\J r_i$ is a reflection.
The arrangements of $\J T_3$ is shown in \Cref{tab:fig:arr-A_4xI} and it contains the arrangement of $T_3$.
The quotient arrangement is obtained again by the one of $T_3$ and by the observation that all planes $e_i^\perp$ are in the same orbit, because of rotation around a line containing a vertex of the tetrahedron sends a reflection plane to another one.
We also note that $\subweyl{\J T_3}=A_1\times A_1\times A_1$, and $\mathcal R_2G=T_3$.


\subsection{The group $Oct_3$}
The linear group $Oct_3$ is generated strictly in codimension $2$.
In particular, $\subweyl G=\triv[3]$, and $\mathcal R_2G=G$.
It is the group of rigid symmetries of the octahedron, i.e., dual polytopes of the cube.
Its subspace arrangement and the respective quotient arrangement are in \Cref{table:oct-3}.


\subsection{The group $\J Oct_3$}
The group $Oct_3$ has six order-two rotations with axes $f_1, \dots, f_6$, and three order two rotations with axes $e_1,e_2,e_3$ respectively.
It follows that $\J Oct_3$ contains nine reflections, across the planes $e_i^\perp$ and $f_i^\perp$.
The subspace arrangement of $\J Oct_3$ is in \Cref{tab:fig:arr-S_4xI}.
By definition, see \Cref{lem:bou-ref}, $BC_3$ is the group of symmetry of a cube, so we observe that $\J Oct_3=BC_3$, $\subweyl G=G$, and $\mathcal R_2G=Oct_3$.

\subsection{The Mixed Group $(Oct_3, T_3)$}
We are going to show that the linear group $(Oct_3, T_3)$ is precisely $A_3$, the symmetry group of the regular tetrahedron $P$. 
Indeed, $T_3\subset (Oct_3,T_3)$, and further $(Oct_3,T_3)$ preserves $P$; by counting elements, we deduce that $(Oct_3,T_3)$ is the symmetry group of $P$.
In particular, we have $\subweyl G=G$ and $\mathcal R_2G=T_3$.

Now, by \Cref{def:mixedgroup}, the group $(Oct_3, T_3)$ contains the elements $\J g$ where $g\in Oct_3\backslash T_3$; by \Cref{lem:FromRotationToReflection} this $\J g$ is a reflection if and only if $g$ is a rotation by an angle $\pi$; these rotations are precisely the six rotations along the lines passing through middle points of edges.
Hence, the arrangement of $(Oct_3, T_3)$, see \Cref{tab:fig:a3}, is the subset of $\mathcal A^{\J Oct_3}$ spanned by $e_i$, $v_i$ and $f_i^\perp$.

%
The quotient arrangement of  $(Oct_3, T_3)$ is a sub-arrangement of the quotient arrangement of $\J Oct_3$.

\subsection{The group $Ico_3$}
The construction of the arrangements for $Ico_3$ is very similar to the one of $Oct_3$ and $T_3$.
These arrangements are shown in \Cref{tab:fig:arr-A_5}.
Here we shortly explain how to get these arrangements, because it will be useful for the next group, $\J Ico_3$.

There are three types of rotations in the
symmetries of the icosahedron. 
The first one are around the lines through opposite vertices $v_1, \dots, v_6$; then, we consider the rotations around the axises through two opposite edges $e_1, \dots, e_{15}$; and finally the rotations around lines through central points of opposite facets $f_1, \dots, f_{10}$.
Hence all those lines are in the arrangements $\mathcal{A}^{Ico_3}$.

To construct the quotient arrangement, we note that there are only three orbits of lines, the $f$-orbits, the $e$-orbits and the $v$-orbits. Indeed, there are three $Ico_3$-orbits in the icosahedron corresponding to vertices, edges, and facets.
Finally this is a group strictly generated in codimension two: $\subweyl G=\triv[3]$, and $\mathcal R_2G=G$.


\subsection{The group $H_3=\J Ico_3$} 
We are going to show that $\J Ico_3$ is the full icosahedral group; in our notation this is identified with the Coxeter group $H_3$, hence strictly generated in codimension $1$.  

As we said above, there are three types of rotations in the icosahedral group $Ico_3$: $v_i$, $e_i$ and $f_i$. (In the previous subsection they are described in details.) 
Only the rotations $e_1, \dots, e_{15}$ have order $2$. So, $\J e_1, \dots, \J e_{15}$ are reflection with respect the planes $\pi_1=e_1^{\perp}, \dots, \pi_{15}^{\perp}$. 
This description, and an order check, already shows that $\J Ico_3$ is the Coxeter group $H_3$ with $\subweyl H_3=H_3$, and $\mathcal R_2 H_3=Ico_3$.

Nevertheless, we want to convince the reader that the two subspace (actually hyperplane) arrangements coincide. The presentation of the hyperplane of $H_3$ is, for instance, in Exercise 26 of Section 7.3 of \cite{GordonMcNulty-Matroids:AGeometric}.

Let us pick a rotation $e_i$: this means picking four vertices among the $12$ of the icosahedron. (Those vertices identify two opposed edges.)
Consider the induced simplicial complex on the remaining eight vertices: This is made of four antipodal $2$-simplexes and two antipodal edges $a$ and $b$. 
We are going to show that each $e_i^\perp$ contains six lines. 
%
%
Indeed, the plane $\pi_i=e_i^\perp$ contains these edges, $a$ and $b$; so it contains the two lines through pairs of antipodal vertices, and lines through the middle points of $a$ and $b$). 
%
The plane $\pi_i$ also contains lines through the middle points of the common edges of the two antipodal $2$-simplexes; in addition, the lines through the pairs of opposite facets of the latter $2$-simples also belongs to $\pi_i$.
These arrangements are shown in \Cref{tab:fig:arr-A_5xI}.

\section{Cohomological computations}
In this section, we compute the singular reduced cohomology of $U_{\mathcal A^G}$, the open complement of the arrangement $\mathcal{A}^G$ in $\mathbb{R}^n$. 
Firstly, we recall the Goresky -- MacPherson formula and we set some useful notations; later in \Cref{sec:cohomology-computations} we show our results 

\subsection{Some Notations and Results from the Literature}
In 1988, Goresky and MacPherson \cite{GoreskyMacPhersonSubspaces} gave a formula to express the group cohomology of the open complement of a real subspace arrangement.

Given a real central subspace arrangement $\mathcal{A}$, we denote by $\mathcal{L}(\mathcal{A})$ its lattice of intersection. 
We set $U_\mathcal{A}\define  \mathbb{R}^n \setminus \cup_{x\in \mathcal{L}(\mathcal{A})} x$, that is the open complement of the arrangement in $\mathbb{R}^n$.
For every, $x\in \mathcal{L}(\mathcal{A})$, 
the interval $[\hat{0}, x]$ is the subposet of $\mathcal{L}(\mathcal{A})$ made by the elements $y\in \mathcal{L}(\mathcal{A})$ such that $\hat{0}\leq y \leq x$.
We are going to denote $\Delta_x\define \roc(\hat{0}, x)$, the (reduced) order complex, see \Cref{sec:posets}.
We mainly care about the lattice of intersection associated to the arrangement $\mathcal{A}^G$ and we are going to show that the simplicial complex 
$$\Delta_G\define \roc (\hat{0}, \hat{1})=\Delta\mathcal{L}(\mathcal{A})$$ 
plays a crucial role. 
For simplicity, we write $U_G$ instead of $U_{\mathcal{A}^G}$.

With an abuse of notation, we denote the rational singular reduced cohomology and the rational simplicial reduced homology by $\operatorname{H}^*(U_\mathcal{A})$ and $\operatorname{H}_{*}(\Delta(\hat{0}, x))$.
We write $\operatorname{h}^*(U_\mathcal{A})$ and $\operatorname{h}_{*}(\Delta(\hat{0}, x))$ for their dimensions as $\mathbb{Q}$-vector space.

We are going to compute the singular reduced cohomology of $U_\mathcal{A}$ via the simplicial reduced homologies of the order complex of the interval $[\hat{0}, x]$ in the lattice $\mathcal{L}(\mathcal{A})$. 

\begin{Theo}(Goresky--MacPherson formula)
	Let $\mathcal{A}$ be a subspace arrangement in $\mathbb{R}^n$ and $U_\mathcal{A}$ be its open complement in $\mathbb{R}^n$. Then 
	\begin{equation*}
		\operatorname{H}^i(U_\mathcal{A})\cong 
        \bigoplus_{\hat{0}\neq x\in \mathcal{L}(\mathcal{A})} \operatorname{H}_{\operatorname{codim}(x)-2-i}\Delta_x.
	\end{equation*}
\end{Theo}

We rewrite the previous formula in a more convenient way; the set of intersections in a specific codimension $j$ is $\mathcal{L}_j(\mathcal{A})$:

\begin{equation}\label{eq:Goresky-MacPherson-formula}
	\operatorname{H}^i(U_\mathcal{A})\cong 
        \bigoplus_{j\neq 0}^n \bigoplus_{x\in \mathcal{L}_j(\mathcal{A})} \operatorname{H}_{j-2-i}\Delta_x,
	\end{equation}
Note that the sum is over $j > 0$ because the is always a unique minimal element $\hat{0}$ corresponding to the intersection $\mathbb{R}^n$. We also remark that $\mathcal{L}_j(\mathcal{A^G})=\mathcal{A}^G_j$.

Moreover, if the arrangement is central, there is also the unique maximal element $\hat{1}$ corresponding to the origin $O$.
In the above direct sum, the contribution for $j=n$ is $\operatorname{H}_{n-2-i}\Delta_{\hat{1}}$. So we observe that $\operatorname{H}^n(U_\mathcal{A})\cong \operatorname{H}_{-2}\Delta_{\hat{1}}=0$.

We say that a subspace arrangement $\mathcal{A}$ is {\em non-trivial} if there is at least an intersection different from $\hat{0}$ and $\hat{1}$, i.e., the lattice of intersection $\mathcal{L}(\mathcal{A})$ is different from $\{\hat{0}<\hat{1}\}$. In other words, $\roc (\hat{0}, \hat{1}))\neq \emptyset$.

\begin{lem}
	A central arrangement of subspaces $\mathcal{A}$ is trivial if and only if $\operatorname{H}^{n-1}(U_\mathcal{A})\neq 0$.
\end{lem}
\begin{proof}
	By equation (\ref{eq:Goresky-MacPherson-formula}), we note that $\operatorname{H}^{n-1}(U_\mathcal{A})\cong\operatorname{H}_{-1}(\roc(\hat{0}, \hat{1}))$. 
    This is non-zero if and only if the order complex $\roc(\hat{0}, \hat{1})$ is not empty.
\end{proof}

\begin{lem}\label{lem-atom-contribution}
	Let $\mathcal{A}$ be a non-trivial central arrangement of subspaces. If $a\in \mathcal{L}(\mathcal{A})$, then $\operatorname{H}_{-1} \Delta_a$ is non zero if and only if $a$ is an atom of the lattice of intersection $\mathcal{L}(\mathcal{A})$.
\end{lem}
\begin{proof}
	We observe that $\operatorname{h}_{-1}\Delta_a=1$ if and only if $\Delta_a$ is empty if and only if $a$ is an atom for $\mathcal{L}(\mathcal{A})$.
\end{proof}

\subsection{The Cohomology of the Principal Stratum}\label{sec:cohomology-computations}

We assume that our arrangement is non-trivial. 
In this case the maximal cohomology could be in degree $(n-2)$.
Using equation (\ref{eq:Goresky-MacPherson-formula}), one gets:
\begin{equation}\label{eq-hom-n-2}
\operatorname{H}^{n-2}(U_\mathcal{A})\cong 
        \operatorname{H}_{0}\Delta_{\hat{1}} \bigoplus_{l} \operatorname{H}_{-1}\Delta_l	
\end{equation}
where the sum is over the lines $l$ in the arrangement.


\begin{pro}\label{Pro-h_n-2-bound}
Let $N$ be the number of atom lines in the arrangements $\mathcal{A}$.
Then $\operatorname{h}^{n-2}(U_{\mathcal{A}})\geq 2N-1$. 
Further, the inequality is strict, i.e., $\operatorname{h}^{n-2}(U_{\mathcal{A}})= 2N-1$ if and only if $\mathcal{A}$ contains precisely $N$ lines.
\end{pro}
\begin{proof}
We know from equation \ref{eq-hom-n-2}, that only $\operatorname{H}_{0}\Delta_{\hat{1}}$ and $\operatorname{H}_{-1}\Delta_l$ contribute to the $(n-2)$-cohomology. 
By \Cref{lem-atom-contribution}, each $\operatorname{H}_{-1}\Delta_l$ is non-zero (and one dimesional) if and only if $l$ is an atom. On the other hand, $\operatorname{H}_{0}\Delta_{\hat{1}}$ could be made by several disconnected components; among those, there is a singleton vertex $\{l\}$ for each atom line $l$. They contribute by an $(N-1)$-dimensional vector space in $\operatorname{H}_{0}\Delta_{\hat{1}}$.
Thus, $\operatorname{h}^{n-2}(U_{\mathcal{A}})= \operatorname{h}_{0}\Delta_{\hat{1}}+ N \geq 2N-1$.

Assume now that $\mathcal{A}$ is only made by lines, so every line is an atom for the lattice of intersection of $\mathcal{A}$. 
Moreover, $\Delta_{\hat{1}}$ has $N$ vertices. 
So  $\operatorname{h}^{n-2}(U_{\mathcal{A}})=2N-1$.
Vice versa, assume $\operatorname{h}^{n-2}(U_{\mathcal{A}})=2N-1$, then $\operatorname{h}^{0}\Delta_{\hat{1}}=N-1$; note that $\Delta_{\hat{1}}$ contains only $N$ vertices.   
\end{proof}

\begin{TheoM}\label{Thm-strictly-two-inGL3}
Let $G$ be a finite linear group in $GL_3(\mathbb{R})$.
The group $G$ is strictly generated in codimension two if and only if the cohomology of $U_G$ is concentrated in degree $1$, and $\operatorname{h}^{1}(U_G)=2N-1$, where $N$ is the number of lines in the arrangements.
\end{TheoM}
\begin{proof}
	Since $G$ is strictly generated in codimesion two, then that $\mathcal{A}^G$ is non-trivial and $\operatorname{h}^{2}(U_G)=0$. 
	By \Cref{pro:no-cod-one}, we have that $\mathcal{A}^G$ is only made by lines. So we apply the \Cref{Pro-h_n-2-bound}.
\end{proof}

Finally, we have enough information to fulfill the three dimensional finite linear groups.

\begin{Theo}\label{Thm-cohomology-GL3}
Let $G$ be a finite linear subgroup of ${GL}_3(\mathbb{R})$. 
\begin{itemize}
    \item If $G$ is strictly generated in codimension $3$, then $\mathcal{A}^G$ is trivial, $U_G=\mathbb{R}^3\setminus O$,  $h^0 (U_G)=h^1 (U_G)=0$ and $h^2 (U_G)=1$;

	\item If $G$ is strictly generated in codimension two, then the arrangement is made by $N$ atom lines, and $h^0 (U_G)=h^2 (U_G)=0$ and $h^1 (U_G)=2N-1$;

	\item If $G$ is {\em strictly} generated in codimension one, then the $G$ is a reflection group, $h^2 (U_G)=0$, $h^1 (U_G)=h_0\Delta_G$, and $h^0 (U_G)= C - 1$, where $C$ is the number of chambers of the sub-hyperplane arrangements in $\mathcal{A}^G$.

    \item If $G$ is generated in codimension two or three, then $h^2 (U_G)=0$, $h^1 (U_G)=h_0\Delta_G + N$, and $h^0 (U_G)= C - 1$, where $C$ is the number of chambers of the sub-hyperplane arrangements in $\mathcal{A}^G$.
\end{itemize}
\end{Theo}
\begin{proof}
The proof uses all the facts we have proved along the section. 
The first item follows easily from the definition of trivial arrangement.
The second item follows from \Cref{Thm-strictly-two-inGL3}.
The case of reflection groups is widely studied, see \cite{Orlik-Terao-book,DeConciniProcesi-Topics-book}.
In this case, let us only mention that there are no atom lines, hence $h^1 (U_G)=h_0\Delta_G$.
We also remark that $h^0 (U_G)+1$ is the number of chambers $C$ of the hyperplane arrangements. 

The last item arises when there are generators in different codimensions. 
The arrangement is not a reflection arrangement, but it might contain atom lines.
Such atom lines change the degree one cohomology, as described by equation \ref{eq-hom-n-2}; the degree zero cohomology is unchanged. 
\end{proof}

The previous theorem suggests two results in higher dimension.

\begin{pro}
Let $G$ be a finite linear group generated in codimension one. Then, $h^0 (U_G)=h^0(U_{\mathcal{R}_1(G)})$.
\end{pro}
\begin{proof}
The proof follows from the same observation given in the proof of the last case of \Cref{Thm-cohomology-GL3}.
\end{proof}

\begin{Cor}
Let $G$ be a finite linear group generated such that $\mathcal{A}^G$ has no hyperplane. Then, $h^0 (U_G)=0$.
\end{Cor}

In the four dimensional real case, we are particularly interested in the finite groups generated in codimension two.
From what we have already proven in this section, $\operatorname{h}^{i}(U_G)=0$ for $i\neq 1,2$.
Let us have a look at the remaining cohomology group:

\begin{equation}\label{eq-hom-n-3}
\operatorname{H}^{1}(U_\mathcal{A})\cong 
        \operatorname{H}_{1}\Delta_{\hat{1}}%
        \bigoplus_{l} \operatorname{H}_{0}\Delta_l%
        \bigoplus_{\pi} \operatorname{H}_{-1}\Delta_{\pi}
\end{equation}

Using \Cref{lem-atom-contribution}, we know that $\bigoplus_{\pi} \operatorname{H}_{-1}(\Delta_{\pi})$ counts the number of planes; all planes are atoms since this the $G$ is strictly generated in codimension two. Set $\sum_{\pi} \operatorname{h}_{-1}(\Delta_{\pi})=a^G_{2}=\#\mathcal{A}^G_2$, the number of intersections in codimension two of $\mathcal{A}^G$. 
More in general, we set
\[
	a_i\define\#\mathcal{A}_i, \mbox{ and } a_i^G\define \#\mathcal{A}^G_i.
\]

\begin{pro}\label{pro-arrangement-with-linee-and-planes}
    %
    If $\mathcal{A}$ is a central subspace arrangement made only by lines and planes in $\mathbb{R}^n$, then 
    \[
    \sum_{l} \operatorname{h}_{0}(\Delta_l)=\sum_{l} M_l -a_{n-1} + N,
    \]
    where $M_l=\#\{\pi\in \mathcal{A}: l\subset \pi \}$ is the number of planes containing $l$, and $N$ is the number of atom lines.
\end{pro}
\begin{proof}
	We note that $a_{n-1}=\#\mathcal{A}_{n-1}$ is the number of lines in the arrangements.
	The order complex $\Delta_l$ is empty if and only if $l$ is a atom lines. In such case $\operatorname{h}_{0}\Delta_l=0$.
    If $l$ is not an atom, then there exist a plane $\pi$ such that $\hat{0}<\pi<l$ and $\Delta_l$ is the union of $M_l$ vertices, precisely, $M_l=\#\{\pi\in \mathcal{A}: l\subset \pi \}$, the number of planes containing $l$; hence $\operatorname{h}_{0}\Delta_l=M_l-1$. 
    We observe that $M_l=0$, if $l$ is an atom. Hence, $\sum_{l} \operatorname{h}_{0}(\Delta_l)=\sum_{l'} (M_{l'}-1)$, where the second sum run over the non atom lines $l'$.
    Observe that $$\sum_{l'} (M_{l'}-1)=\sum_{l'} M_{l'}-\sum_{l'} 1=\sum_{l} M_{l}-\sum_{l'} 1.$$  
    Observe that $a_{n-1}=\sum_{l} 1=\sum_{l'} 1+N$.
\end{proof}

The finite linear groups generated in codimension two in ${GL}_{3}(\mathbb{R})$ and strictly generated in codimension two in ${GL}_{4}(\mathbb{R})$ satisfy the hypothesis of the previous proposition.




\begin{Cor}
If $G$ is a finite linear group strictly generated in codimension two in ${GL}_{4}(\mathbb{R})$, then\[
    \sum_{l} \operatorname{h}_{0}(\Delta_l)=\sum_{l} M_l -a^G_{3} + N.
\]
\end{Cor}
\begin{proof}
There are no hyperplanes in $\mathcal{A}^G$ because of \Cref{pro:no-cod-one}. Observe that lines are in codimension three.
\end{proof}

If $G$ is strictly generated in codimension two in ${GL}_{4}(\mathbb{R})$, then 
\begin{equation}\label{eq:h_n-3-open}
\operatorname{h}^{1}(U_G)= \operatorname{h}_{1}\Delta_{G} + \sum_{l} M_l -a^G_3 + N + a^G_2.
\end{equation}
where $M_l$, $a^G_2$, $a^G_3$, and $N$ are defined above.

\begin{Theo}\label{THM:cohomology-dim4}
If $G$ is a finite linear group strictly generated in codimension two in ${GL}_{4}(\mathbb{R})$, then
\begin{eqnarray*}
    \operatorname{h}^{1}(U_G)&=& \operatorname{h}_{1}\Delta_{G} + \sum_{l} M_l -a^G_3 + N + a^G_2,\\
    \operatorname{h}^{2}(U_G)&>&2N-1.
\end{eqnarray*}
\end{Theo}

\addcontentsline{toc}{section}{Bibliography}
\bibliographystyle{amsalpha}
\bibliography{biblio}

\vspace{0.5cm}
 
\noindent
 {\scshape Ivan Martino}\\
 {\scshape Department of Mathematics, Northeastern University,\\ Boston, MA 02115, USA}.\\
 {\itshape E-mail address}: \texttt{i.martino@northeastern.edu}

\vspace{0.5cm}

\noindent
 {\scshape Rahul Singh}\\
 {\scshape Department of Mathematics, Northeastern University,\\ Boston, MA 02115, USA}.\\
 {\itshape E-mail address}: \texttt{singh.rah@husky.neu.edu}
\end{document}